\documentclass[11pt,twoside,reqno]{amsart}
\usepackage{amsmath, amsfonts, amsthm, amssymb, graphicx, cite}

\usepackage{hyperref}
\usepackage{esint}
\usepackage{mathrsfs}
\usepackage{tikz}
\usetikzlibrary{matrix,arrows,decorations.pathmorphing}
\usepackage{setspace}
\usepackage{epigraph}

\usepackage{mathtools}
\usepackage{tikz,tikz-cd}
\usetikzlibrary{matrix,arrows,decorations.pathmorphing}

\newcommand{\R}{\mathbb{R}}

\newcommand{\N}{\mathbb{N}}

\DeclarePairedDelimiter\abs{\lvert}{\rvert}

\DeclarePairedDelimiter\cbrac{\{}{\}}
\DeclarePairedDelimiter\sqbrac{[}{]}
\DeclarePairedDelimiter\brac{(}{)}

\theoremstyle{plain}
\newtheorem{theorem}{Theorem}[section]
\newtheorem{lemma}{Lemma}[section]
\newtheorem{proposition}{Proposition}[section]

\theoremstyle{problem}

\theoremstyle{definition}
\newtheorem{definition}{Definition}[section]

\theoremstyle{remark}
\newtheorem{remark}{Remark}[section]

\numberwithin{equation}{section}

\begin{document}

\title[Stability of Steady Multi-Wave Configurations for the Full Euler Equations]
{Stability of Steady Multi-Wave Configurations for the Full Euler Equations of Compressible Fluid Flow}

\author
{Gui-Qiang G. Chen \qquad Matthew Rigby}
\address{Gui-Qiang G. Chen,
Mathematical Institute, University of Oxford,
Oxford, OX2 6GG, UK}
\email{chengq@maths.ox.ac.uk}
\address{Matthew Rigby,
Mathematical Institute, University of Oxford,
Oxford, OX2 6GG, UK}
\email{rigby@maths.ox.ac.uk}

\dedicatory{{\it In memoriam} Professor Xiaqi Ding}

\date{\today}
\thanks{}

\begin{abstract}
We are concerned with the stability of steady multi-wave configurations for the full Euler equations
of compressible fluid flow. In this paper, we focus on the stability of steady four-wave configurations
that are the solutions of the Riemann problem in the flow direction,
consisting of two
shocks, one vortex sheet, and one entropy wave,
which is one of the core multi-wave configurations for
the two-dimensional Euler equations.
It is proved that such steady four-wave configurations in supersonic flow are stable
in structure globally, even under the BV perturbation of the incoming flow in the flow direction.
In order to achieve this, we first formulate the problem as
the Cauchy problem (initial value problem) in the flow direction,
and then develop a modified Glimm difference scheme and
identify a Glimm-type functional to obtain the required BV estimates by tracing the interactions
not only between the
strong shocks and weak waves, but also between the strong vortex sheet/entropy wave and weak waves.
The key feature of the Euler equations is that the reflection coefficient is always less than $1$,
when a weak wave of different family interacts with
the strong vortex sheet/entropy wave or the shock wave,
which is crucial to guarantee that the Glimm functional is decreasing.
Then these
estimates are employed to establish the convergence of the approximate solutions to a global entropy
solution, close to the background solution of steady four-wave configuration.
\end{abstract}
\maketitle

\section{Introduction}
We are concerned with the stability of steady multi-wave configurations for the two-dimensional steady full Euler equations
of compressible fluid flow governed by
\begin{equation} \label{eq:euler}
\begin{cases} (\rho u)_x + (\rho v)_y = 0, \\
		(\rho u^2 + p)_x + (\rho u v )_y = 0, \\
		(\rho u v)_x + (\rho v^2 + p)_y = 0, \\
		(\rho u (E + \frac{p}{\rho}))_x + (\rho v(E + \frac{p}{\rho}))_y = 0,
\end{cases}
\end{equation}
where $(u,v)$ is the velocity, $\rho$ the density, $p$ the scalar pressure, and
$
E = \frac 12 (u^2 + v^2) + e(p,\rho)
$
the total energy, with internal energy $e$ that is a given function of $(p,\rho)$ defined through thermodynamic relations.
The other two thermodynamic variables are the temperature $T$ and the entropy $S$.
If $(\rho,S)$ are chosen as two independent variables, then the constitutive relations become
\begin{equation} \label{eq:epT}
	(e,p,T) = (e(\rho,S), p(\rho,S), T(\rho,S)),
\end{equation}
governed by
\begin{equation} \label{eq:governed}
	T dS = de - \frac{p}{\rho^2} \, d \rho.
\end{equation}
For an ideal gas,
\begin{equation} \label{eq:idealgas}
	p = R\rho T, \qquad e = c_v T, \qquad \gamma = 1 + \frac{R}{c_v} > 1,
\end{equation}
and
\begin{equation} \label{eq:pressureandenergy}
p = p(\rho,S) = \kappa \rho^{\gamma} e^{\frac{S}{c_v}},
\qquad e = \frac{\kappa}{\gamma-1} \rho^{\gamma -1} e^{\frac{S}{c_v}}
= \frac{RT}{\gamma-1},
\end{equation}
where $R, \kappa$, and $c_v$ are all positive constants.
The quantity
\[
c = \sqrt{p_{\rho}(\rho,S)} = \sqrt{\frac{\gamma p}{\rho}}
\]
is defined as the sonic speed.

In this paper, we focus on the stability of steady four-wave configurations
in the two space-dimensional case,
consisting of two shocks, one vortex sheet, and one entropy wave,
which are the solutions of the Riemann
problem in the flow direction; see Figure 1.
In this configuration, the vortex sheet and the entropy wave coincide in the Euler coordinates.
This is one of the fundamental core multi-wave configurations, as a solution of the
standard steady Riemann problem
for the two-dimensional Euler equations:

\smallskip
(i) For supersonic flow, there are at most eight waves (shocks, vortex sheets, entropy waves, rarefaction waves) that
emanate from one single point in the Euler coordinates,
which consist of one solution (at most four of these waves) of the Riemann problem in the flow direction
and the other solution (at most four of these waves) of the other Riemann problem in the opposite direction,
while the later Riemann problem can also be reduced into the standard Riemann problem in the flow direction
by the coordinate transformation $(x,y)\to (-x,-y)$ and the velocity transformation $(u,v)\to (-u,-v)$, which
are invariant for the Euler equations \eqref{eq:euler}.

\smallskip
(ii) Vortex sheets and entropy waves are new key fundamental waves in the multidimensional case,
which are normally very sensitive in terms of perturbations as observed in numerical simulations
and physical experiments ({\it cf.} \cite{AM87,AM89,Chen17,cf-book2014shockreflection,CW2012}).

\smallskip
(iii) Such solutions are fundamental configurations for the local structure of general
entropy solutions, which play an essential role in the mathematical theory of
hyperbolic conservation laws
({\it cf.} \cite{bressan,CH,Chen88,cf-book2014shockreflection,dafermos,Ding1,Ding2,Ding3,glimm,Liu,smoller}).

\begin{figure}[h!]
	\includegraphics[scale=0.8]{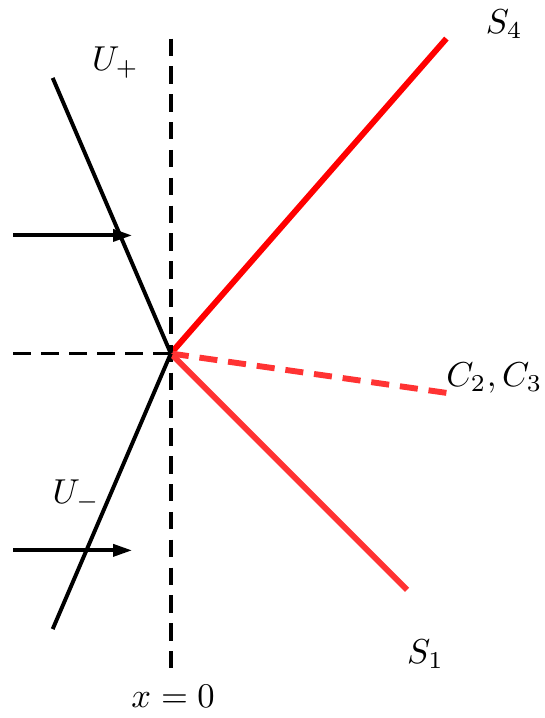}
	\centering
	\caption{An unperturbed four-wave configuration, consisting of two shocks $S_1$ and $S_4$, one vortex sheet $C_2$, and one entropy wave $C_3$}
\end{figure}

The stability problem involving supersonic flows with a single shock past a Lipschitz wedge
has been solved in Chen-Zhang-Zhu \cite{shockwedges} (also see Chen-Li \cite{ChenLi2008}).
The stability problem involving supersonic flows with vortex sheets and entropy waves
over a Lipschitz wall has been solved in Chen-Zhang-Zhu \cite{compvortex}.
See also Chen-Kuang-Zhang \cite{ChenKuangZhang17}
for the stability of two-dimensional steady supersonic exothermically reacting Euler flow past
Lipschitz bending walls.

The case of an initial configuration involving two shocks is treated in \cite{lewicka},
by using the method of front tracking,
for more general equations, under the finiteness and stability conditions.
We think that, with the estimates of Riemann solutions involving more than two strong waves,
the estimates on the reflection coefficients of wave interactions
should play a similar role so that the method of front tracking may be used.

In this paper, it is proved that steady four-wave configurations in supersonic flow are stable
in structure globally, even under the BV perturbation of the incoming flow in the flow direction.
In order to achieve this, we first formulate the problem as
the Cauchy problem (initial value problem) in the flow direction, then develop a modified Glimm difference scheme
similar to those in \cite{shockwedges,compvortex} from the original Glimm scheme in \cite{glimm}
for one-dimensional hyperbolic conservation laws, and further
identify a Glimm-type functional to obtain the required BV estimates by tracing the interactions
not only between the
strong shocks and weak waves, but also between the strong vortex sheet/entropy wave and weak waves carefully.
The key feature of the Euler equations is that the reflection coefficient is always less than $1$,
when a weak wave of different family interacts with the vortextsheets/entropy wave or the shock wave,
which is crucial to guarantee that the Glimm functional is decreasing.
Then these
estimates are employed to establish the convergence of the approximate solutions to a global entropy
solution, close to the background solution of steady four-wave configuration.

This paper is organized as follows: In \S 2,  we first formulate the stability of multi-wave configurations as
the Cauchy problem (initial value problem) in the flow direction for the Euler equations (\ref{eq:euler})
and then state the main theorem of this paper.
In \S 3, some fundamental properties of system \eqref{eq:euler} and the analysis of
the Riemann solutions are presented, which are used in the subsequent sections.
In \S 4,  we make estimates on the wave interactions, especially between the strong
and weak waves, and identify the key feature of the Euler equations that
the reflection coefficient is always less than $1$,
when a weak wave of different family interacts with
the vortex sheet/entropy wave or the shock wave.
In \S 5, we develop a modified Glimm difference scheme, based on the ones in \cite{shockwedges,compvortex},
to construct a family of approximate solutions, and establish necessary estimates that will
be used later to obtain its convergence to an entropy solution of the Cauchy problem \eqref{eq:euler} and \eqref{eq:initdata}.
In \S 6, we show the convergence of the approximate solutions to an entropy solution,
close to the background solution of steady four-wave configuration.

\section{Formulation of the Problem and  Main Theorem}

In this section, we formulate the stability problem for the steady four-wave configurations as
the Cauchy problem (initial value problem) in the flow direction for the Euler equations (\ref{eq:euler})
and then state the main theorem of this paper.

\subsection{Stability problem}

We focus on the stability problem of the four-wave configurations consisting of two strong shocks, one strong vortex sheet, and one entropy wave
for the supersonic Euler flows governed by system (\ref{eq:euler}) for $U = (u,v,p,\rho)$.
More precisely, we consider a background solution $\overline U: \R^2_+ \rightarrow \R$ that consists of four constant states:
\begin{align*} \label{eq:backgroundsoln}
	&U_b = (u_b,v_b, p_b, \rho_b), \\
	&U_{m1} = (u_{m1},0,p_{m1},\rho_{m1}), \\
	&U_{m2} = (u_{m2},0,p_{m2},\rho_{m2}), \\
	&U_a = (u_a,v_a,p_a,\rho_a),
\end{align*}
where $u_j > c_j$ for all $j \in \{a,m_1,m_2,b \}$ with the sonic speed  of state $U_j$:
$$
c_j = \sqrt{\frac{\gamma p_j}{\rho_j}},
$$
and state $U_{m1}$ connects to $U_b$ by a strong $1$--shock of speed $\sigma_{10}$,
$U_{m1}$ connects to $U_{m2}$ by a strong $2$--vortex sheet and a strong $3$--entropy wave
of strengths $(\sigma_{20},\sigma_{30})$, and $U_{a}$ connects to $U_{m2}$ by a strong $4$--shock
of speed $\sigma_{40}$; see Figure \ref{fig:backgroundsoln}.

\begin{figure}[h!]
	\includegraphics{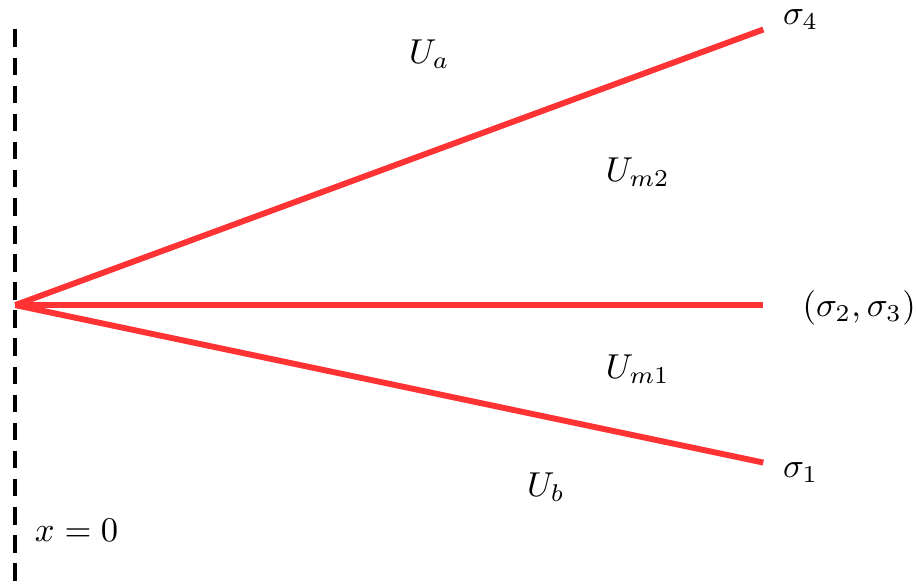}
	\centering
	\caption{The background solution $\overline{U}$, consisting of four waves and four constant states}
	\label{fig:backgroundsoln}
\end{figure}

We are interested in the stability of the background solution $\overline{U}$
of steady four-wave configuration, under small $BV$ perturbations of
the incoming flow as the initial data,
to see whether it leads to entropy solutions containing similar strong four-wave configurations,
close to the background solution $\overline{U}$.
That is, the stability problem can be formulated as the following
Cauchy problem (initial value problem) for the Euler equations \eqref{eq:euler}
with the Cauchy data:
\begin{equation}\label{eq:initdata}
	\left. U \right| _{x=0} = U_0,
\end{equation}
where  $U_0 \in \mathrm{BV}(\R)$
is a small perturbation function close to $\overline U(\cdot,0)$ in $BV$.

The main theorem of this paper is the following:

\begin{theorem}[Existence and Stability]\label{thm:mainthm}
There exist $\epsilon> 0$ and $C>0$ such that,
if $U_0 \in BV(\R)$ satisfies
\[
\mathrm{TV}\{U_0(\cdot)  - \overline U(0,\cdot)\}< \epsilon,
\]
then there are four functions{\rm :}
\[
U \in BV_{loc}(\R^2_+) \cap L^\infty(\R^2_+), \qquad\, \chi_i \in \mathrm{Lip}(\R_+;\R_+) \quad\mbox{for $i =1,2,3,4$,}
\]
such that
\begin{enumerate}
\item[\rm (i)] $U$ is a global entropy solution of system \eqref{eq:euler} in $\R^2_+$,
  satisfying the initial condition \eqref{eq:initdata}{\rm ;}
\item[\rm (ii)] Curves $\{ y = \chi_i(x) \}$, $i=1, 2, 3, 4$, are a strong $1$--shock, a combined strong $2$--vortex sheet and $3$--entropy wave ($\chi_{2,3}:=\chi_2=\chi_3$),
and a strong $4$--shock, respectively, all emanating from the origin, with
\begin{align*}
		&\abs*{U(x,y) |_{\{y < \chi_1(x)\}} - U_b}  < C \epsilon,  \\
		&\abs*{U(x,y) |_{\{\chi_{1}(x) <  y < \chi_{2,3}(x)\}} - U_{m1} }  < C \epsilon,  \\
		&\abs*{U(x,y) |_{\{\chi_{2,3}(x) <  y < \chi_4(x)\}} - U_{m2}}   < C \epsilon, \\
		&\abs*{U(x,y) |_{\{\chi_4(x) < y\}} - U_a}  < C \epsilon.
\end{align*}
\end{enumerate}
\end{theorem}

In \S 2--\S6, we prove this main theorem and related properties of the global solution in $BV$.

\section{Riemann Problems and Solutions}

This section includes some fundamental properties of system \eqref{eq:euler}
and some analysis of the Riemann solutions, which will be used
in the subsequent sections; see also \cite{shockwedges, compvortex}.

\subsection{Euler equations}

With $U = (u,v,p,\rho)$, the Euler system can be written in the following conservation form:
\begin{equation} \label{eq:consform1}
W(U)_x + H(U)_y = 0,
\end{equation}
where
\begin{equation}\label{eq:consform}
\begin{aligned}
W(U) = (\rho u, \rho u^2 + p, \rho u v, \rho u(h + \frac{u^2 + v^2}{2})) , \,\,
H(U) = (\rho v, \rho u v, \rho v^2 + p, \rho v (h + \frac{u^2 + v^2}{2})),
\end{aligned}
\end{equation}
and $h = \frac{\gamma p}{(\gamma -1) \rho} $.
For a smooth solution $U(x,y)$, (\ref{eq:consform}) is equivalent to
\begin{equation}
\nabla_U W(U) U_x + \nabla_U H(U) U_y = 0,
\end{equation}
so that the eigenvalues of \eqref{eq:consform} are
the roots of the fourth order polynomial{\rm :}
\begin{equation}
\mathrm{det}\brac*{\lambda \nabla_U W(U) - \nabla_U H(U)},
\end{equation}
which are solutions of the equation:
\begin{equation}
(v-\lambda u)^2 \brac*{(v-\lambda u)^2 - c^2 (1 + \lambda^2)} = 0,
\end{equation}
where $c = \sqrt{\gamma p/\rho}$ is the sonic speed.
If the flow is supersonic, {\it i.e.} $u^2 + v^2 > c^2$,
system (\ref{eq:euler}) is hyperbolic.
In particular, when $u > c$, system (\ref{eq:euler}) has the following four
eigenvalues in the $x$--direction:
\begin{equation}\label{eq:eigenvalues}
\lambda_j = \frac{uv + (-1)^j c \sqrt{u^2 + v^2 - c^2}}{u^2 - c^2} \quad\mbox{for $j = 1,4$}; \,\,\qquad \lambda_i = \frac{v}{u} \quad \mbox{for $i = 2,3$},
\end{equation}
with four corresponding linearly independent eigenvectors:
\begin{equation}\label{eq:eigenvectors}
\begin{aligned}
&\mathbf{r}_j= \kappa_j (-\lambda_j, 1,\rho(\lambda_j u - v), \frac{\rho(\lambda_j u - v)}{c^2})^\top \qquad \mbox{for $j= 1,4$}, \\[2mm]
&\mathbf{r}_2 = (u,v,0,0)^\top, \qquad \mathbf{r}_3 = (0,0,0,\rho)^\top,
\end{aligned}
\end{equation}
where $\kappa_j$ are chosen to ensure that $\mathbf{r}_j \cdot \nabla \lambda_j = 1$ for $j=1,4$,
since the first and fourth characteristic fields are always genuinely nonlinear,
and the second and third are linearly degenerate.

In particular, at a state $U=(u,0,p,\rho)$,
\[
\lambda_2(U) = \lambda_3(U) = 0, \qquad \lambda_1(U) = - \frac{c}{\sqrt{u^2 - c^2}} = - \lambda_4(U) < 0.
\]
	
\begin{definition} $U \in BV(\R^2_+)$ is called an entropy solution of (\ref{eq:euler}) and (\ref{eq:initdata}) if
\begin{enumerate}
\item[(i)]  $U$ satisfies the Euler equations  (\ref{eq:euler}) in the distributional sense and (\ref{eq:initdata}) in the trace sense;
\item[(ii)] $U$ satisfies the following entropy inequality:
			\begin{equation} \label{eq:entropyineq} (\rho u S)_x + (\rho v S)_y \leq 0 \end{equation}
			in the distributional sense in $\R^2_+$, including the boundary.
\end{enumerate}
\end{definition}
	
\subsection{Wave curves in the phase space}

In this subsection, based on  \cite[pp. 297--298]{shockwedges} and \cite[pp. 1666--1667]{compvortex},
we look at the basic properties of nonlinear waves.

We focus on the region, $\{u > c\}$, in the state space,
especially in the neighborhoods of $U_j$ in the background solution.

We first consider self-similar solutions of (\ref{eq:euler}):
\[
(u,v,p,\rho)(x,y) = (u,v,p,\rho)(\xi), \qquad \xi = \frac{y}{x},
\]
which connect to a state $U_0 = (u_0,v_0,p_0,v_0)$. We find that
\begin{equation}\label{eq:eigens}
\mathrm{det}\big(\xi \nabla_U W(U) - \nabla_U H(U)\big) = 0,
\end{equation}
which implies
\[
\xi = \lambda_i = \frac{v}{u} \quad\mbox{for $i= 2,3$},
\quad \text{or} \quad \xi = \lambda_j = \frac{uv + (-1)^j c \sqrt{u^2 + v^2 - c^2}}{u^2 - c^2} \quad \mbox{for $j = 1,4.$}
\]

First, for the cases $i=2,3$, we obtain
\begin{equation}\label{eq:vortexsheets}
		dp = 0, \quad v du - u dv = 0.
\end{equation}
This yields the following curves $C_i(U_0)$ in the phase space through $U_0$:
\begin{equation}\label{eq:vortex}
C_i(U_0): \quad p = p_0, \quad w =\frac{v}{u} = \frac{v_0}{u_0}\,\, \qquad \mbox{for $i =2,3$},
\end{equation}
which describe compressible vortex sheets $(i=2)$ and entropy waves $(i=3)$.
More precisely, we have a vortex sheet governed by
\begin{equation} \label{eq:param2sheet}
C_2(U_0): \, U = (u_0 e^{\sigma_2}, v_0 e^{\sigma_2}, p_0, \rho_0)^\top
\end{equation}
with strength $\sigma_2$ and slope $\frac{v_0}{u_0}$, which is determined by
\[
	\frac{d U}{d \sigma_2} = \mathbf{r_2}(U), \qquad
	U|_{\sigma_2 = 0} = U_0;
\]
and an entropy wave governed by
\begin{equation}\label{eq:param3sheet}
C_3(U_0): \, U = (u_0 , v_0 , p_0, \rho_0 e^{\sigma_3})^\top
\end{equation}
with strength $\sigma_3$ and slope $\frac{v_0}{u_0}$, which is determined by
\[
	\frac{d U}{d \sigma_3} = \mathbf{r_3}(U), \qquad
	U|_{\sigma_3 = 0} = U_0.
\]

For $j=1,4$, we obtain the $j$th rarefaction wave curve $R_j(U_0)$, $j=1,4$, in the phase space through $U_0$:
\begin{equation}\label{eq:14rarefaction}
R_j(U_0):\, dp = c^2 d \rho,\, du = -\lambda_j dv, \,\rho(\lambda_j u - v) dv = dp \,\qquad \text{for } \rho < \rho_0, u >c , \, j =1,4.
\end{equation}
For shock wave solutions, the Rankine-Hugoniot conditions for \eqref{eq:euler} are
\begin{eqnarray}
&& s \sqbrac*{\rho u} = \sqbrac*{\rho v}, \label{eq:rankine1}\\
&& s \sqbrac*{\rho u^2 + p} = \sqbrac*{\rho u v}, \label{eq:rankine2}\\
&& s \sqbrac*{\rho u v } = \sqbrac*{\rho v^2 + p},\label{eq:rankine3}\\
&& s \big[\rho u \big(h + \frac{u^2+ v^2}{2}\big)\big] = \big[\rho v \big(h + \frac{u^2+v^2}{2}\big)\big],
\label{eq:rankine4}
\end{eqnarray}
where the jump symbol $\sqbrac*{\,\cdot\,}$ stands for the value of the front state minus that of the back state.
We find that
\[
(v_0 - s u_0)^2 \big( (v_0 - su_0)^2 - \overline c^2 (1+s^2)\big) = 0,
\]
where $\overline c^2 = \frac{c_0^2}{b}$ and $b = \frac{\gamma + 1}{2} - \frac{\gamma -1 }{2} \frac{\rho}{\rho_0}$. This implies
\begin{equation}\label{eq:23shockspeed}
s = s_i = \frac{v_0}{u_0} \qquad \mbox{for $i=2,3$},
\end{equation}
or
\begin{equation}\label{eq:14shockspeed}
 s = s_j = \frac{u_0 v_0 + (-1)^j \overline c \sqrt{u_0^2 + v_0^2 - \overline c^2}}{u_0^2 - \overline c^2} \qquad \mbox{for $j= 1,4$},
\end{equation}
where $u_0 > \overline c$ for small shocks.
	
For $s_i$, $i=2,3$, in (\ref{eq:rankine1})--(\ref{eq:rankine4}),
we obtain the same $C_i(U_0)$, $i =2,3$,  defined in \eqref{eq:param2sheet}--\eqref{eq:param3sheet},
since the corresponding fields are linearly degenerate.

On the other hand, for $s_j, j =1,4$, in (\ref{eq:rankine1})--(\ref{eq:rankine4}),
we obtain the $j$th shock wave curve $S_j(U_0)$, $j=1,4$, through $U_0$:
\begin{equation} \label{eq:14shocksdescribed}
S_j(U_0): \,\sqbrac*{p} = \frac{c_0^2}{b} \sqbrac*{\rho}, \,\sqbrac*{u} = - s_j \sqbrac*{v},\, \rho_0 (s_j u_0 - v_0) \sqbrac*{v} = \sqbrac*{p}
\qquad\, \text{for $\rho > \rho_0, u > c, \,\, j =1,4$},
\end{equation}
where $\rho_0 < \rho$ is equivalent to the entropy condition (\ref{eq:entropyineq}) on the shock wave.
We also know that $S_j(U_0)$ agrees with $R_j(U_0)$ up to second order and that
\begin{equation}\label{eq:densityineq}
\left.	\frac{d \sigma_1}{d \rho} \right|_{S_1(U_0)} < 0, \qquad \left.	\frac{d \sigma_4}{d \rho} \right|_{S_4(U_0)} > 0.
\end{equation}
The entropy inequality (\ref{eq:entropyineq}) is equivalent to the following:
\begin{equation}\label{eq:shockspeeds}
\begin{aligned}
&\lambda_j(\text{above})  < \,  \sigma_j < \lambda_j(\text{below}) \qquad\mbox{for $j=1,4$},\\
& \sigma_1  < \lambda_{2,3}(\text{below}), \\
&		\lambda_{2,3}(\text{above}) < \,   \sigma_4;
\end{aligned}
\end{equation}
see \cite[pp. 269--270, pp. 297--298]{shockwedges} for the details.

\subsection{Riemann problems}
We consider the Riemann problem for \eqref{eq:euler}:
\begin{equation}\label{eq:rproblem}
\left.	U \right| _{x=x_0}
= \begin{cases} U_a, &\,\,\, y>y_0,\\
			U_b, &\,\,\, y< y_0,
\end{cases}
\end{equation}
where $U_a$ and $U_b$ are constant states, regarded as the above and below state
with respect to line $y=y_0$.
	
\subsubsection{Riemann problem only involving weak waves}
Following Lax \cite{lax}, we can parameterize any physically admissible wave curve in a neighborhood of a constant state $U_0$.
	
\begin{lemma}
Given $U_0\in \R^4$, there exists a neighborhood $O_\epsilon(U_0)$ such that, for all $U_1,U_2 \in O_\epsilon(U_0)$,
the Riemann problem \eqref{eq:rproblem} admits a unique admissible solution consisting of four elementary waves.
In addition, state $U_2$ can be represented by
\begin{equation}\label{3.25a}
U_2 = \Phi(\alpha_4,\alpha_3,\alpha_2,\alpha_1;U_1)
\end{equation}
with
\begin{align*}
& \Phi(0,0,0,0;U_1) = U_1,\\
& \partial_{\alpha_i} \Phi(0,0,0,0;U_1) = \mathbf{r}_i(U_1) \qquad\mbox{for $i=1,2,3,4$}.
\end{align*}
\end{lemma}

From now on, we denote $\{ U_1, U_2 \} = (\alpha_1,\alpha_2,\alpha_3,\alpha_4)$ as a compact way to write the representation
of \eqref{3.25a}.

We also note that the renormalization factors $\kappa_j$ in \eqref{eq:eigenvectors}
have been used to ensure that $\mathbf r_j \cdot \nabla \lambda_j = 1$
in a neighborhood of any unperturbed state $U_0 = (u_0, 0, p_0, \rho_0)$ with $u_0 > 0$, such as $U_{m1}$ or $U_{m2}$:
	
\begin{lemma}
At any state $U_0 = (u_0, 0, p_0,\rho_0)$ with $u_0 > 0$,
\[
\kappa_1(U_0) = \kappa_4(U_0) > 0,
\]
which also holds  in a neighborhood of $U_0$.
\end{lemma}

Also, since $\nabla_U \Phi(0,0,0,0;U_1)$ is equal to the identity,
by the implicit function theorem,
we can  find $\tilde \Phi$ (after possibly shrinking $O_\epsilon(U_0)$) such that, in the above situation,
we can represent
\[
U_1 = \tilde \Phi(\alpha_4,\alpha_3,\alpha_2,\alpha_1;U_2).
\]
Differentiating the relation:
\[
\tilde \Phi(\alpha_4,\alpha_3,\alpha_2,\alpha_1; \Phi(\alpha_4,\alpha_3,\alpha_2,\alpha_1;U_1)) = U_1
\]
and using that
\[
\tilde \Phi(0,0,0,0;U_2) = U_2,
\]
we deduce
\[
\partial_{\alpha_i} \tilde \Phi(0,0,0,0;U_2) = - \mathbf{r}_i(U_2)  \qquad \mbox{for $i=1,2,3,4$}.
\]
	
This will be used later in \S 4.
We exploit the symmetries between the shock polar and the reverse shock polar,
and the symmetry between the $1$-shock polar and the reverse $4$-shock polar to allow for more concise arguments.

\subsubsection{Riemann problem involving a strong $1$--shock}

The results here are based on those in \S 6.1.4 of \cite{shockwedges}, with small changes for our requirements.
	
For a fixed $U_1$, when $U_2 \in S_1(U_1)$, we use $\{ U_1, U_2 \} = (\sigma_1,0,0,0)$ to denote the $1$--shock that
connects $U_1$ to $U_2$ with speed $\sigma_1$.

\begin{lemma}\label{lem:shock1speed}
For all $U_1 \in O_\epsilon(U_{b}), U_2 \in S_1(U_1)\cap O_{\epsilon}(U_{m1})$, and $\sigma_{1} \in O_{\hat \epsilon}(\sigma_{10})$ with $\{ U_1, U_2 \} = (\sigma_1,0,0,0)$,
\[
\sigma_1 < 0, \qquad u_1 < u_2 < \big(1 + \frac{1}{\gamma}\big) u_1.
\]
\end{lemma}
	
\begin{proof}
We follow the same steps as \cite[pp. 273]{shockwedges}.
First, by \eqref{eq:14shockspeed},
\[
\sigma_{10} = - \frac{\overline c}{\sqrt{u_{m1}^2 - \overline c^2}} < 0.
\]
Now, from the Rankine-Hugoniot conditions \eqref{eq:rankine1}--\eqref{eq:rankine4},
\begin{align}
&\sigma_{10}( \rho_b u_b - \rho_{m1} u_{m1}) = \rho_b v_b, \label{eq:rkl1} \\
&\sigma_{10}\big( \rho_b u_b^2 - \rho_{m1}u^2_{m1} + p_b - p_{m1}\big) = \rho_b v_b u_b.
\label{eq:rkl2}
\end{align}
Thus, using \eqref{eq:rkl1}--\eqref{eq:rkl2},
\[
p_{m1} - p_b = \rho_{m1} u_{m1} (u_b - u_{m1}),
\]
so that $u_b > u_{m1}$ as $p_{m1} > p_b$.
Also, since $p_{b} > 0$, we have
\[
\frac{1}{\gamma}\rho_{m1} c_{m1}^2 =   p_{m1} > \rho_{m1} u_{m1} (u_b - u_{m1}),
\]
so that,  as $u_{m1} > c_{m1}$,
\[
u_b < \big(1+ \frac{1}{\gamma}\big) u_{m1}.
\]		
The result now follows by continuity.
\end{proof}
	
\begin{lemma} \label{lem:matrixadet} \label{lem:matrixa}
Let
\[
A : = \nabla_U H(U_{m1}) - \sigma_{10} \nabla_U W(U_{m1}),
\]
and
\[
P := u_{m1} \Big(h_b + \frac{u_{b}^2 + v_b^2}{2} - \frac{u_{m1}^2}{2}\Big) + \Big(\frac{c_{m1}^2}{\gamma-1} + u_{m1}^2\Big)(u_{m1} - u_b).
\]
Then
\begin{align*}	
&\mathrm{det} A > 0,\quad  \mathrm{det}(A \mathbf{r_4}, A \mathbf{r_3}, A \mathbf{r_2}, A \mathbf{r_1})|_{U = U_{m1}} > 0 , \\
&\mathrm{det}(A \mathbf{r_4}, A \mathbf{r_3}, A \mathbf{r_2}, A \partial_{\sigma_1} G_1(\sigma_{10},U))|_{U = U_{m1}} > 0.
\end{align*}
\end{lemma}
	
\begin{proof}
We omit the full calculations, which can be found at \cite[pp. 299--300]{shockwedges}.
We use Lemma \ref{lem:shock1speed}, along with the first and fourth Rankine-Hugoniot
conditions \eqref{eq:rankine1} and \eqref{eq:rankine4},
to deduce that $P>0$.
We then use Lemma \ref{lem:shock1speed} and the entropy condition \eqref{eq:shockspeeds}
to deduce
\begin{align*}
&\mathrm{det} A = \frac{\sigma_{10}^2 \rho_{m1}^3 u_{m1}^3}{\gamma-1} \brac*{\lambda_1^2(U_{m1}) - \sigma_{10}^2} \brac*{u^2_{m1} - c_{m1}^2} > 0, \\
&\mathrm{det}(A \mathbf{r_4}, A \mathbf{r_3}, A \mathbf{r_2}, A \partial_{\sigma_1} G_1(\sigma_{10},U))|_{U = U_{m1}} \\
&\quad = \frac{\kappa_4(U_{m1})^2\sigma_{10} \rho_{b} v_b \rho_{m1}^3 u_{m1}^2 }{ \lambda_4(U_{m1})}
  \big(\sigma_{10} - \lambda_4(U_{m1})\big)\big( u_{m1} \lambda_4(U_{m1}) P + \sigma_{10} u_{m1} Q\big) > 0,
\end{align*}
where $Q = - \frac{c_+^2}{\gamma-1} < 0$.
\end{proof}

\begin{lemma}\label{lem:paramstrong1}
There exists a neighborhood $O_\epsilon(U_{b}) \times O_{\hat \epsilon}(\sigma_{10})$ such that,
for each $U_0 \in O_\epsilon(U_b)$, the shock polar $S_1(U_b)$ can be parameterized locally for
the state which connects to $U_0$ by a shock of speed $\sigma_{10}$ from above as
\[
\sigma_1 \mapsto G_1(\sigma_1, U_0)\,\, \qquad \text{for } \,\, |\sigma_1 - \sigma_{10}| < \hat \epsilon.
\]
\end{lemma}

\begin{proof} It suffices to solve
\[
\sigma_1 \big(W(U) - W(U_0)\big) - \big(H(U) - H(U_0)\big) = 0
\]
for $U$ in terms of $\sigma_1$ and $U_0$,
with the knowledge that $(U_0,\sigma_1,U) = (U_b, \sigma_{10},U_{m_1})$ is a solution.
We see that
\begin{align*}
\mathrm{det}  \big(\nabla_U(\sigma (W(U) - W(U_0)) - (H(U) - H(U_0)))\big)
\bigg|_{U_0 = U_b,\sigma_1=\sigma_{10}, U = U_{m1} }= \mathrm{det}A > 0.
\end{align*}
Then the result follows by the implicit function theorem.
\end{proof}
	
\subsubsection{Riemann problem involving a strong $4$--shock}

We now extend our results about $1$--shocks to $4$--shocks by symmetry.
For a fixed $U_1$, when $U_2 \in S_4(U_1)$,
we use $\{ U_1, U_2 \} = (0,0,0,\sigma_4)$ to denote the $4$--shock that connects $U_1$
to $U_2$ with speed $\sigma_4$.
The only difference is the formula for $\sigma_4$.

\begin{lemma}\label{lem:shock4speed}
For all $U_1 \in O_\epsilon(U_{m2}), U_2 \in O_\epsilon(U_a) \cap S_1(U_1)$, and $\sigma_{3} \in O_{\hat \epsilon}(\sigma_{40})$
with $\{ U_1, U_2 \} = (0,0,0,\sigma_4)$,
\[
\sigma_{4} > 0, \qquad u_1 < u_2 < \big(1 + \frac{1}{\gamma}\big) u_1.
\]
\end{lemma}

\begin{proof}
Note that, by \eqref{eq:shockspeeds},
\[
\sigma_4 = \frac{u_a v_a +  \overline c \sqrt{u_a^2 + v_a^2 - \overline c^2}}{u_a^2 - \overline c^2} > 0.
\]
Since
we have the same Rankine-Hugoniot conditions
as in Lemma \ref{lem:shock1speed},
with $u_{m2}$ and $u_a$ taking the roles of $u_{m1}$ and $u_b$, respectively,
the proof follows identically.
\end{proof}

\begin{lemma}\label{lem:paramstrong4}
There exists a neighborhood $O_\epsilon(U_{a}) \times O_{\hat \epsilon}(\sigma_{40})$
such that,
for each $U_0 \in O_\epsilon(U_a)$, the reverse shock polar $\tilde S_2(U_0)$ -- the set of states that
connect to $U_0$ by a strong $4$--shock from below -- can be parameterized locally for
the state which connects to $U_0$ by a shock of speed $\sigma_{40}$ as
\[
\sigma_4 \mapsto G_4(\sigma_4, U_0),
\]
with $G_4 \in C^2$ near $(\sigma_{40},U_{m2})$.
\end{lemma}
	
\subsubsection{Riemann problem involving strong vortex sheets and entropy waves}

We now look at the interaction between weak waves and the strong vortex sheet/entropy wave,
based on those in \S 2.5 of \cite{compvortex}.
For any $U_1 \in O_\epsilon(U_{m1})$ and $U_2 \in O_\epsilon(U_{m2})$,
we use $\{ U_1, U_2 \} = (0,\sigma_2,\sigma_3,0)$ to denote the strong vortex sheet and entropy wave
that connect $U_1$ to $U_2$ with strength $(\sigma_2,\sigma_3)$.
That is,		
\[
U_{\rm mid} = \Phi_2(\sigma_2,U_1) = (u_1 e^{\sigma_2},v_1 e^{\sigma_2}, p_1, \rho_1),
\,\,  U_2 = \Phi(\sigma_3, U_{\rm mid}) = (u_{\rm mid},v_{\rm mid},p_{\rm mid},\rho_{\rm mid} e^{\sigma_3}).
\]
In particular, we have
\[
U_{2} = (u_{1} e^{\sigma_{2}}, v_1 e^{\sigma_2}, p_1, \rho_1 e^{\sigma_3}).
\]
By a straightforward calculation, we have

\begin{lemma}\label{lem:vortexcalc}
For
\[
G_{2,3}(\sigma_3,\sigma_2,U)
:= \Phi_3(\sigma_3; \Phi_2(\sigma_2; U)) = (e^{\sigma_{20}}u,e^{\sigma_{20}}v, p, e^{\sigma_{30}}\rho)
\qquad \mbox{for any $U \in O_\epsilon(U_{m1})$},
\]
then
\begin{align*}	
\partial_{\sigma 2} G_{2,3}(\sigma_3,\sigma_2;U) &= (u e^{\sigma_2}, v e^{\sigma_2},0, 0)^\top, \\
			\partial_{\sigma 3} G_{2,3}(\sigma_3,\sigma_2;U) &= (0, 0 , 0 , \rho e^{\sigma_3} )^\top, \\
			\nabla_U G_{2,3} (\sigma_3,\sigma_2,U) &= \mathrm{diag}( e^{\sigma_2}, e^{\sigma_2}, 1 , e^{\sigma_3}).
\end{align*}
\end{lemma}

The next property allows us to estimate the strength of reflected weak waves
in the interactions between the strong vortex sheet/entropy wave and weak waves:
	
\begin{lemma}\label{lem:vortexdet}
The following holds{\rm :}
\[
\mathrm{det}\brac* { \mathbf r_4(U_{m2}), \partial_{\sigma_3} G_{2,3}(\sigma_{30},\sigma_{20}; U_{m1}), \partial_{\sigma_2} G_{2,3}(\sigma_{30},\sigma_{20}; U_{m1}),
 \nabla_U G(\sigma_{30}, \sigma_{20}, U_{m1}) \cdot \mathbf{r}_1(U_1)  } > 0.
 \]
\end{lemma}
	
\begin{proof} By a direct calculation,
\begin{align*}	
&\mathrm{det}\brac* { \mathbf r_4(U_{m2}), \partial_{\sigma_3} G_{2,3}(\sigma_{30},\sigma_{20}; U_{m1}),
  \partial_{\sigma_2} G_{2,3}(\sigma_{30},\sigma_{20}; U_{m1}), \nabla_U G(\sigma_{30}, \sigma_{20}, U_{m1}) \cdot \mathbf{r}_1(U_1)} \\[2mm]
&= \kappa_1(U_{m1}) \kappa_4(U_{m2})
\left| \begin{array}{cccc}  \lambda_4(U_{m2}) & 0 & u_{m1} e^{\sigma_{20}} & - e^{\sigma_{20}}\lambda_1(U_{m1}) \\[2mm]
				1 & 0 & 0 & e^{\sigma{20}} \\[2mm]
				\rho_{m2} u_{m2} \lambda_4(U_{m2}) & 0 & 0 & \rho_{m1} u_{m1} \lambda_1(U_{m1}) \\[2mm]
				\frac{\rho_{m2} u_{m2} \lambda_4(U_{m2})}{c_{m2}^2} & \rho_{m1} e^{\sigma_{30}} & 0 & e^{\sigma_{30}} \frac{\rho_{m1} u_{m1} \lambda_1(U_{m1})}{c_{m1}^2}
             \end{array} \right| \\[2mm]
&= \kappa_1(U_{m1}) \kappa_4(U_{m2}) \rho_{m1}^2 u_{m1}^2 e^{\sigma_{20} + \sigma_{30}} \brac*{\lambda_4(U_{m2}) e^{2 \sigma_{20}
 + \sigma_{30}} + \lambda_4(U_{m1}) } > 0.
\end{align*}
\end{proof}

\section{Estimates on the Wave Interactions}\label{chapter:waves}
	
In this section, we make estimates on the wave interactions,
especially between the strong and weak waves.
This is based on those in \S  3 of  \cite{shockwedges,compvortex}, with new estimates for the strong $4$--shock.
		
Below, $M>0$ is a universal constant which is understood to be large,
and $O_\epsilon(U_i)$ for $i \in \{a, m1, m2, b\}$ is a universal small neighborhood of $U_i$
which is understood to be small.
Each of them depends only on the system, which may be different at each occurrence.
	
\begin{figure}[h!]
\includegraphics[scale=1.0]{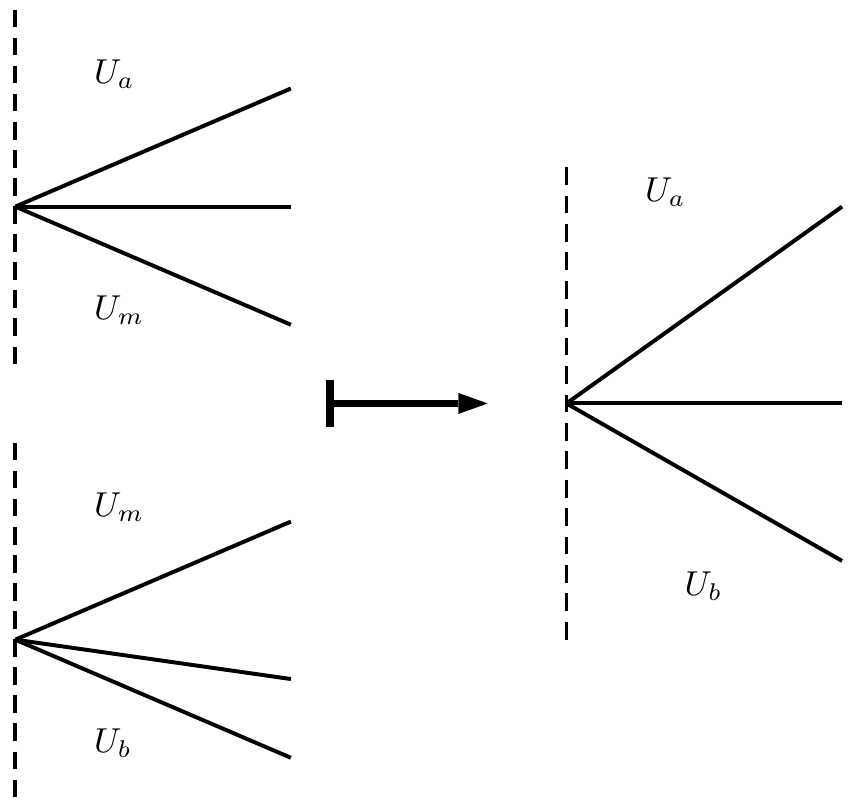}
\centering
\caption{Interaction estimates}
\label{fig:interactiondiagram}
\end{figure}

\subsection{Preliminary identities}
	
To make later arguments more concise, we now state some elementary identities here to be used later;
these are simple consequences of the fundamental theorem of calculus.
	
\begin{lemma}\label{lem:1varest}
The following identities hold{\rm :}	
\begin{itemize}
\item	If $f \in C^1([a,b])$ with $a < 0 < b$, then, for any $x \in [a,b]$,
\[
f(x) - f(0)  =  x \int_0^1 f_x(rx) \, dr;
\]
\item
If $g \in C^2([a,b]^2)$ with $a < 0 < b$, then, for any $(x,y) \in [a,b]^2$,
\[
f(x,y) - f(x,0) - f(0,y) + f(0,0) = xy \int_0^1 \int_0^1 f_{xy}(rx,sy) \, dr \, ds;
\]
\item If $f \in C^1([a,b])$, then
\[
f(x) = f(0) + O(1) |x|;
\]
\item	If $f \in C^2( [a,b]^3)$, then
\[
f(x,y,z) = f(x,0,0) + f(0,y,z) - f(0,0,0) +  O(1) |x| ( |y| + |z|);
\]
\item 	If $g \in C^2([a,b]^4)$, then
\[
g(x,y,z,w) = g(x,0,0,0) + g(0,y,z,w) - g(0,0,0,0) +  O(1) |x|( |y| + |z| + |w|).
\]
\end{itemize}
\end{lemma}
	
\subsection{Estimates on weak wave interactions}
We have the following standard proposition;
see, for example, \cite[Chapter 19]{smoller} for the proof.
Note that, in our analysis of the Glimm functional,
we only require the estimates for the cases where the waves are approaching.

\begin{proposition}\label{prop:weakwaves}
Suppose that $U_1,U_2$, and $U_3$ are three states in a small neighborhood of a given state $U_0$ with
		\begin{align*}
			\{ U_1,U_2 \} &= (\alpha_1,\alpha_2,\alpha_3,\alpha_4), \\
			\{ U_2, U_3 \} &= (\beta_1,\beta_2,\beta_3,\beta_4), \\
			\{ U_1,U_3 \} &= (\gamma_1,\gamma_2,\gamma_3,\gamma_4).
		\end{align*}
Then
\begin{equation}\label{eq:weakwaves}
\gamma_i = \alpha_i + \beta_i + O(1) \Delta(\alpha,\beta)
\end{equation}
with
\[
\Delta(\alpha,\beta)
= |\alpha_4| |\beta_1| + |\alpha_3||\beta_1| + |\alpha_2||\beta_1| + |\alpha_4||\beta_2| + |\alpha_4| |\beta_3|
+ \sum_{j=1,4} \Delta_j(\alpha,\beta),
\]
where
\[
\Delta_j(\alpha,\beta)
= \begin{cases} 0, &\,\, \alpha_j \geq 0, \beta_j \geq 0, \\
		|\alpha_j| |\beta_j|, &\,\, \text{otherwise}.
\end{cases}
\]
\end{proposition}
	
\subsection{Estimates on the interaction between the strong vortex sheet/entropy wave and weak waves}

We now derive an interaction estimate between the strong vortex sheet/entropy wave and weak waves.
The properties that $|K_{vb1}| < 1$ in \eqref{eq:vortexests}
and $|K_{va4}|<1$ in \eqref{4.5-b}
will be critical in the proof that
the Glimm functional is decreasing.

\begin{proposition}\label{prop:belowvortex}
Let $U_1, U_2 \in O_{\epsilon}(U_{m1})$ and $U_3 \in O_{\epsilon}(U_{m2})$ with
\[
\{ U_1, U_2 \} = (0,\alpha_2,\alpha_3,\alpha_4) \qquad
\text{and} \qquad \{U_2,U_3\} = (\beta_1,\sigma_2,\sigma_3,0),
\]
there is a unique $(\delta_1,\sigma'_2,\sigma'_3,\delta_4)$ such that
\[
\{ U_1, U_3 \} = (\delta_1,\sigma'_2,\sigma'_3,\delta_4),
\]
and
\begin{align*}
\delta_1 &=  \beta_1 + K_{vb1} \alpha_4 + O(1) \Delta', \\
\sigma'_2 &= \sigma_2 + \alpha_2  + K_{vb2} \alpha_4 + O(1) \Delta', \\
\sigma'_3 &= \sigma_3 + \alpha_3 + K_{vb3} \alpha_4 + O(1) \Delta' ,\\
\delta_4 &= K_{vb4} \alpha_4 + O(1) \Delta',
\end{align*}
with
\begin{equation}\label{eq:vortexests}
 \Delta' = |\beta_1|( |\alpha_2| + |\alpha_3|), \qquad  |K_{vb1}| <1, \qquad \sum_{j=2}^4 |K_{vbj}| \leq M.
\end{equation}
\end{proposition}

\begin{proof}
This proof is the same as \cite[pp. 1673]{compvortex}, with additional terms.
We need to solve for $(\delta_1,\sigma_2',\sigma_3',\delta_4)$ as a function
of $(\alpha_1,\alpha_2,\alpha_3,\alpha_4, \beta_1,\sigma_2,\sigma_3,\beta_4, U_1)$  in the following equation:
\begin{equation}\label{eq:belowvortexeqn}
\Phi_4(\delta_4; G_{2,3}(\sigma_2',\sigma_3';\Phi_1(\delta_1;U_1)))
= \Phi_4(\beta_4;G(\sigma_3,\sigma_2; \Phi_1(\beta_1; \Phi(\alpha_4,\alpha_3,\alpha_2,\alpha_1;U_1)))).
\end{equation}

Lemma \ref{lem:vortexdet} implies
\begin{align*}
&\left.	\mathrm{det} \brac*{\frac{\partial \Phi_4(\delta_4; G_{2,3}(\sigma_2',\sigma_3';\Phi_1(\delta_1;U_1))) }{\partial(\delta_4,\sigma_3',\sigma_2',\delta_1)} }
  \right|_{\delta_4 = \delta_1 = 0, \sigma_2' = \sigma_{20}, \sigma_3' = \sigma_{30}, U_1 = U_{m1} } \hfill \\
&= \mathrm{det}\brac* { \mathbf r_4(U_{m2}), \partial_{\sigma_3} G_{2,3}(\sigma_{30},\sigma_{20}; U_{m1}),
 \partial_{\sigma_2} G_{2,3}(\sigma_{30},\sigma_{20}; U_{m1}), \nabla_U G(\sigma_{30}, \sigma_{20}, U_{m1}) \cdot \mathbf{r}_1(U_1)  }\\
& > 0. \hfill
\end{align*}

Note that, when $\alpha_4 = 0$, the unique solution is
\[
\delta_1 = \beta_1, \quad \sigma'_2  = \sigma_2 + \alpha_2, \quad \sigma'_3 = \sigma_3 + \alpha_3,
\quad \delta_4 = \beta_4,
\]
so that
\begin{equation}\label{eq:waveests}
\begin{aligned}
\delta_1 &= \beta_1 + K_{vb1} \alpha_4 + \alpha_1 +  O(1) \Delta', \\
\delta_4 & = K_{vb4} \alpha_4  + \beta_4 + O(1) \Delta', \\			
\sigma'_i &= \sigma_i + \alpha_i + K_{vbi} \alpha_4 + O(1) \Delta' \qquad\mbox{for $i=2,3$},
\end{aligned}
\end{equation}
where
\begin{align*}
K_{vbi} = \int_0^1 \partial_{\alpha_4} \sigma'_i(\alpha_2,\alpha_3, \theta \alpha_4, \beta_1, \sigma_2,\sigma_3 ) \, d \theta \qquad\mbox{for $i=2,3$},			\\
K_{vbj} = \int_0^1 \partial_{\alpha_4} \delta_j(\alpha_2,\alpha_3, \theta \alpha_4, \beta_1, \sigma_2,\sigma_3 ) \, d \theta \qquad\mbox{for $j=1,4$}.					
\end{align*}
We see that $K_{vbi}$ and $K_{bvj}$ are bounded from the formulae above.
To deduce that $|K_{vb1}| < 1$, we differentiate \eqref{eq:belowvortexeqn} with respect to $\alpha_4$
to obtain
\begin{align*}
&\nabla_U G_{2,3}(\sigma_{30},\sigma_{20};U_{m1}) \cdot \mathbf{r}_4(U_{m1}) \\
&= \partial_{\alpha_4}\delta_4\, \mathbf r_4(U_{m2}) + \partial_{\alpha_4}\sigma'_3\, \partial_{\sigma_3} G_{2,3} (\sigma_{30},\sigma_{20}; U_{m1})
  + \partial_{\alpha_4}\sigma'_2\, \partial_{\sigma 2} G_{2,3} (\sigma_{30},\sigma_{20}; U_{m1}) \\
& \quad   + \partial_{\alpha_4} \delta_1\, \nabla_U G_{2,3}(\sigma_{30},\sigma_{20};U_{m1}) \cdot \mathbf{r}_1(U_{m1}).
\end{align*}
By Lemma \ref{lem:vortexcalc} and another similar calculation, we have
\begin{align*}
&|\partial_{\alpha_4} \delta_1|\\[1mm]
&= \abs*{ \frac{\mathrm{det}\brac* { \mathbf r_4(U_{m2}), \partial_{\sigma_3} G_{2,3}(\sigma_{30},\sigma_{20}; U_{m1}),
   \partial_{\sigma_2} G_{2,3}(\sigma_{30},\sigma_{20}; U_{m1}), \nabla_U G(\sigma_{30}, \sigma_{20}, U_{m1}) \cdot \mathbf{r}_4(U_1)}}
    { \mathrm{det}\brac* { \mathbf r_4(U_{m2}), \partial_{\sigma_3} G_{2,3}(\sigma_{30},\sigma_{20}; U_{m1}),
     \partial_{\sigma_2} G_{2,3}(\sigma_{30},\sigma_{20}; U_{m1}), \nabla_U G(\sigma_{30}, \sigma_{20}, U_{m1}) \cdot \mathbf{r}_1(U_1)  } } } \\[2mm]
&= \abs*{ \frac{\kappa_1(U_{m1}) \kappa_4(U_{m2}) \rho_{m1}^2 u_{m1}^2 e^{\sigma_{20}
 + \sigma_{30}} \brac*{\lambda_4(U_{m2}) e^{2 \sigma_{20} + \sigma_{30}}
  - \lambda_4(U_{m1}) } }{\kappa_1(U_{m1}) \kappa_4(U_{m2}) \rho_{m1}^2 u_{m1}^2 e^{\sigma_{20}
  + \sigma_{30}} \brac*{\lambda_4(U_{m2}) e^{2 \sigma_{20} + \sigma_{30}} + \lambda_4(U_{m1}) } }} \\[2mm]
& =  \abs*{ \frac{\lambda_4(U_{m2}) e^{2 \sigma_{20} + \sigma_{30}} - \lambda_4(U_{m1})  }{\lambda_4(U_{m2}) e^{2 \sigma_{20} + \sigma_{30}} + \lambda_4(U_{m1}) }} \\[2mm]
&< 1.
\end{align*}
\end{proof}

We can then deduce the following result by symmetry:
\begin{proposition} \label{prop:abovevortex}
If $U_1 \in O_{\epsilon}(U_{m2})$ and $U_2, U_3 \in O_\epsilon(U_{m1})$ with
\[
\{ U_1,U_2 \} = (0,\sigma_2,\sigma_3,\alpha_4), \qquad \{U_2,U_3 \} = (\beta_1,\beta_2,\beta_3,0),
\]
then	
\[
\{ U_1, U_3 \} = (\delta_1, \sigma'_2,\sigma'_3,\delta_4)
\]
with
\begin{align*}
			\delta_1 &=  K_{va1} \beta_1 + O(1) \Delta'',\\
			\sigma'_2 &= \sigma_2 + \beta_2 + K_{va2} \beta_1 + O(1) \Delta'', \\
			\sigma'_3 &= \sigma_3 + \beta_3 + K_{va3} \beta_1 + O(1) \Delta'', \\
			\delta_4 &=\alpha_4+  K_{va4} \beta_1 + O(1) \Delta '',
\end{align*}
and
\begin{equation}\label{4.5-b}
\Delta'' = |\alpha_4|(|\beta_2| + |\beta_3| ), \qquad \sum_{j=1}^3 |K_{vaj}| \leq M, \qquad |K_{va4}| < 1.
\end{equation}
\end{proposition}

\subsection{Estimates on the interaction between the strong shocks and weak waves}

We now derive an interaction estimate between the strong shocks and weak waves.
The properties that $|K_{1sa4}| < 1$ in \eqref{4.6a}
and $|K_{4sb1}|<1$ in \eqref{4.13a}
will be critical in the proof that
the Glimm functional is decreasing.

\subsubsection{Interaction between the strong $1$--shock and weak waves}
	
\begin{proposition}\label{prop:strong1aboveest}
If $U_1 \in O_\epsilon(U_b)$ and $U_2,U_3 \in O_\epsilon(U_{m2})$ with
\[
\{  U_1, U_2 \} = (\sigma_1,\alpha_2,\alpha_3,\alpha_4), \qquad \{U_2,U_3 \} = (\beta_1,\beta_2,\beta_3,0),
\]
then there exists a unique $(\sigma'_1,\delta_2,\delta_3,\delta_4)$ such that
\[
\{ U_1, U_3 \} = (\sigma'_1,\delta_2,\delta_3,\delta_4).
\]
Moreover,
\begin{align*}
&\sigma'_1 = \sigma_1 + K_{1sa1} \beta_1 + O(1) \Delta'' , \qquad\qquad\, \delta_2 = \alpha_2 + \beta_2 + K_{1sa2}\beta_1 + O(1) \Delta'',\\
&\delta_3 = \alpha_3 + \beta_3 + K_{1sa3} \beta_1+ O(1) \Delta '' , \qquad \delta_4 = \alpha_4  + \beta_4 + K_{1sa4} \beta_1 +  O(1) \Delta'' ,
\end{align*}
with
\begin{equation}\label{4.6a}
\Delta'' = |\alpha_4|(|\beta_1| + |\beta_2| + |\beta_3|), \qquad |K_{1sa4}|< 1, \qquad \sum_{j=1}^3 |K_{1saj}| \leq M.
\end{equation}
\end{proposition}

\begin{proof}
This is similar to \cite[pp. 303]{shockwedges}, with extra terms. We need to show that there exists a solution $(\sigma_1',\delta_2,\delta_3,\delta_4)$ to
\begin{equation} \label{eq:1shockaboveeqn}
\Phi(0,\beta_3,\beta_2,\beta_1; \Phi(\alpha_4,\alpha_3,\alpha_2,0; G_1(\sigma_1;U_1))) = \Phi(\delta_4,\delta_3,\delta_2, 0 ; G_1(\sigma'_1,U_1)).
\end{equation}
By Proposition \ref{prop:weakwaves}, there exists $(\gamma_4,\gamma_3,\gamma_2,\gamma_1)$ as a function of $(\alpha_2,\alpha_3,\alpha_4,\beta_1,\beta_2,\beta_3,\beta_4)$ such that
\begin{equation} \label{eq:1shockgammaeqn}
\Phi(0,\beta_3,\beta_2,\beta_1; \Phi(\alpha_4,\alpha_3,\alpha_2,0; G_1(\sigma_1;U_1))) = \Phi(\gamma_4,\gamma_3,\gamma_2,\gamma_1;G_1(\sigma_1;U_1))
\end{equation}
with
\begin{equation}\label{eq:gammas}
\begin{aligned}
&\gamma_1 = \beta_1 + O(1) \tilde \Delta, \qquad\qquad\,\, \gamma_2 = \alpha_2 + \beta_2 + O(1) \tilde \Delta, \\
&\gamma_3 = \alpha_3 + \beta_3 + O(1) \tilde \Delta,  \qquad \gamma_4 = \alpha_4 +  O(1) \tilde \Delta,
\end{aligned}
\end{equation}
where
\[
\tilde \Delta = \Delta'' + |\beta_1|(  |\alpha_2| + |\alpha_3| + |\alpha_4|).
\]
Thus, we may reduce \eqref{eq:1shockaboveeqn} to
\begin{equation}\label{eq:1shockaboveeqn2}
\Phi(\gamma_4,\gamma_3,\gamma_2,\gamma_1;G_1(\sigma_1,U_1))
= \Phi(\delta_4,\delta_3,\delta_2, 0 ; G_1(\sigma'_1,U_1)).
\end{equation}
By Lemma \ref{lem:matrixa},
\begin{align*}
&\mathrm{det} \left.\brac*{ \frac{\partial \Phi(\delta_4,\delta_3,\delta_2, 0 ; G_1(\sigma'_1,U_1)) }{\partial(\sigma_1',\delta_2,\delta_3,\delta_4) }}
  \right|_{ \{U_1 = U_b, \delta_2 = \delta_3 =\delta_4 = 0, \sigma'_1 = \sigma_{10} \}} \\
&= \frac{1}{\mathrm{det} A} \mathrm{det}(A \mathbf{r_4}, A \mathbf{r_3}, A \mathbf{r_2}, A  \partial_{\sigma_1}G_1(\sigma_{10};U_b) )|_{U = U_{m1}}  > 0.
\end{align*}
Next, observe that, when $\gamma_1 = 0$, the unique solution is $(\sigma', \delta_2,\delta_3,\delta_4)=(\sigma_1, \gamma_2,\gamma_3,\gamma_4)$.
Thus, for $i = 2,3,4$,
\begin{equation}\label{eq:deltasigma}
\sigma'_1 = K_{1sa1} \gamma_1 + \sigma_1, \qquad
\delta_i = K_{1sai} \gamma_1 + \gamma_i,
\end{equation}
where
\begin{equation}
\begin{aligned}
K_{1sa1} &= \int_0^1 \partial_{\gamma_1}\sigma'_1(t \gamma_1,\gamma_2,\gamma_3,\gamma_4,\sigma_1,U_1) \, dt, \\
K_{1sai} &= \int_0^1 \partial_{\gamma_1}\delta_i(t \gamma_1,\gamma_2,\gamma_3,\gamma_4,\sigma_1,U_1) \, dt \qquad\mbox{for $i=2,3,4$}.
\end{aligned}
\end{equation}
To see that $|K_{1sa4}|<1$, we find that, by Lemma \ref{lem:matrixa} and another similar computation,
\begin{align*}
|\partial_{\gamma_1} \delta_4|
&=  \abs*{ \frac{\mathrm{det}(A \mathbf r_1, \mathbf A r_2,  A \mathbf r_3, A \partial_{\sigma_1} G_1(\sigma_{10},U_{b}) )}
  {\mathrm{det}(A \mathbf r_4, \mathbf A r_2,  A \mathbf r_3, A \partial_{\sigma_1} G_1(\sigma_{10},U_{b}) )} } \\
&= \abs*{ \frac{ \sigma_{10} + \lambda_1(U_{m1}) }{\sigma_{10} - \lambda_1(U_{m1}) } }
  \abs*{  \frac{ u_{m1} \lambda_4(U_{m1}) P + \frac{c_{m2}^2}{\gamma-1} \sigma_{10} u_b}{ u_{m1} \lambda_4(U_{m1}) P - \frac{c_{m2}^2}{\gamma-1} \sigma_{10} u_b}} < 1.
\end{align*}
Now, using the expression of $\gamma$ in terms of $(\alpha,\beta)$
and absorbing the residual part of the $\tilde \Delta$--term into the $K_{1sai}\beta_1$--term,
we have the desired result.
\end{proof}
	
We now prove a result for the case where weak waves approach the strong $1$-shock from below.

\begin{proposition}\label{prop:strong1belowest}
Suppose that $U_1,U_2 \in O_\epsilon(U_b)$ and $U_3 \in O_\epsilon(U_{m2})$ such that
\[
\{ U_1,U_2 \} = (\alpha_1,\alpha_2,\alpha_3,\alpha_4), \qquad \{ U_2,U_3 \} = (\sigma_1, \beta_2, \beta_3,\beta_4 ).
\]
Then there exists a unique $(\sigma'_1,\delta_2,\delta_3,\delta_4)$ such that
\[
\{ U_1, U_3 \} = (\sigma'_1,\delta_2,\delta_3,\delta_4).
\]
Moreover,
\begin{align*}
\sigma'_1 = \sigma_1 + \sum_{j=1}^4 K_{1sb1j} \alpha_j, \qquad\,\, \delta_2 = \beta_2 +  \sum_{j=1}^4 K_{1sb2j} \alpha_j,   \\
\delta_3 = \beta_3 +  \sum_{j=1}^4 K_{1sb3j} \alpha_j,  \qquad
\delta_4 = \beta_4 +  \sum_{j=1}^4 K_{1sb4j} \alpha_j,
\end{align*}
where  \[\sum_{i,j} |K_{1sb ij}| \leq M.\]
\end{proposition}
	
\begin{proof}
We need to find a solution to
\begin{equation}
\Phi(\delta_4,\delta_3,\delta_2,0;G_1(\sigma'_1;U_1)) = \Phi(\beta_4,\beta_3,\beta_2,0;G_1(\sigma_1; \Phi(\alpha_4,\alpha_3,\alpha_2,\alpha_1;U_1))).
\end{equation}
By Lemma \ref{lem:matrixa},
\begin{align*}
\mathrm{det} \brac*{  \frac{\partial \Phi(\delta_4,\delta_3,\delta_2,0;G_1(\sigma'_1;U_1)) }{\partial(\sigma'_1, \delta_2,\delta_3,\delta_4)}} = \mathrm{det} A > 0.
\end{align*}
Now the required result follows by using Lemma \ref{lem:1varest}, and noting that
$\sigma'_1 = \sigma_1$, $\alpha_1 = \alpha_2 = \alpha_3 = \alpha_4 =0$
and that $(\sigma'_1, \delta_2, \delta_3, \delta_4)=(\sigma_1, \beta_2,\beta_3,\beta_4)$
is a solution.
\end{proof}
	
\subsubsection{Interaction between the strong $4$--shock and weak waves}

Now, by symmetry with the $1$--shock case, we deduce the following results:
\begin{proposition}\label{prop:strong4aboveest}
Suppose that $U_1 \in O_\epsilon(U_{m1})$ and $U_2,U_3 \in O_\epsilon(U_{m2})$ such that
\[
\{ U_1,U_2 \} = (\alpha_1,\alpha_2,\alpha_3,\sigma_4), \qquad \{ U_2,U_3 \} = (\beta_1, \beta_2, \beta_3,\beta_4 ).
\]
Then there exists a unique $(\delta_1,\delta_2,\delta_3,\sigma'_4)$ such that
\[
\{ U_1, U_3 \} = (\delta_1,\delta_2,\delta_3,\sigma'_4).
\]
Moreover,
\begin{align*}
\delta_1 = \alpha_1 + \sum_{j=1}^4 K_{4sa1j} \beta_j , \qquad \delta_2 = \alpha_2 + \sum_{j=1}^4 K_{4sa2j} \beta_j,   \\
\delta_3 = \alpha_3 + \sum_{j=1}^4 K_{4sa3j} \beta_j,   \qquad
\sigma'_4 = \sigma_4 + \sum_{j=1}^4 K_{4sa4j} \beta_j,
\end{align*}
where  $\sum_{i,j} |K_{4sa ij}| \leq M$.
\end{proposition}

\begin{proposition} \label{prop:strong4belowest}
If $U_1,U_2 \in O_\epsilon(U_{m1})$ and $U_3 \in O_\epsilon(U_{a})$ with
\[
\{  U_1, U_2 \} = (0,\alpha_2,\alpha_3,\alpha_4), \qquad \{U_2,U_3 \} = (\beta_1,\beta_2,\beta_3,\sigma_4),
\]
then there exists a unique $(\delta_1,\delta_2,\delta_3,\sigma'_4)$ such that
\[
\{ U_1, U_3 \} = (\delta_1,\delta_2,\delta_3,\sigma'_4).
\]
Moreover,
\begin{align*}
&\delta_1 =  \beta_1 + K_{4sb1} \alpha_4 + O(1) \Delta''',
    \qquad\qquad\,\,   \delta_2 = \beta_2 + \alpha_2 + K_{4sb2}\alpha_4 + O(1) \Delta''', \\
&\delta_3 = \beta_3 + \alpha_3 + K_{4sb3} \alpha_4 + O(1) \Delta ''',
	 \qquad\, \sigma'_4 = \sigma_4  + K_{4sb4} \alpha_4 +  O(1) \Delta''',
\end{align*}
where
\begin{equation}\label{4.13a}
\Delta''' = |\beta_1|(|\alpha_2| + |\alpha_3|) , \qquad |K_{4sb1}|<1, \qquad \sum_{j=1}^3 |K_{4sbj}|  \leq M.
\end{equation}
\end{proposition}

\section{Approximate Solutions and  BV Estimates}

In this section, we develop a modified Glimm difference scheme,
based on the ones in \cite{compvortex,shockwedges}, to construct a family of approximate solutions
and establish necessary estimates that will be used later to obtain its convergence to
an entropy solution to (\ref{eq:euler}) and (\ref{eq:initdata}).

\subsection{A modified Glimm scheme}
For $\Delta x > 0$, we define
\[
\Omega_{\Delta x} = \bigcup_{k \geq 1} \Omega_{\Delta x, k},
\]
where
\[
\Omega_{\Delta x, k} = \{ (x,y)\, :\, (k-1) \Delta x < x \leq k \Delta x \}.
\]
We choose $\Delta y > 0$ such that the following condition holds:
\begin{align*}
\frac{\Delta y}{ \Delta x}  > \max
\Big\{\max_{1\le i\le 4} \sup_{ \sigma \in  O_{\hat \epsilon}(\sigma_{i0})} |\sigma|,
\,\max_{j=1,4} \, \big\{\sup_{U \in O_{\epsilon}(U_a) \cup O_\epsilon(U_b)} |\lambda_j(U)|\big\},
 \,\sup_{U \in O_{\epsilon}(U_{m1}) \cup O_{\epsilon}(U_{m2})} |\lambda_2(U)|\Big\}.
\end{align*}

We also need to make sure that the strong shocks do not interact: If we take the neighborhoods small enough,
there exist $a$ and $b$ with $ -1 < a < 0 < b < 1$ such that
\[
\sigma_{1} <  \frac{\Delta y}{\Delta x} a < \lambda_2(U)  <  \frac{\Delta y}{\Delta x} b < \sigma_4
\]
for all $U \in O_{\epsilon}(U_{m1}) \cup O_{\epsilon}(U_{m2})$ and $\sigma_i \in O_{\hat \epsilon}(\sigma_{i0}), i = 1,4$.
Thus, a strong $1$--shock emanating from $((k-1) \Delta x, n \Delta y)$ meets $x = k \Delta x$ in the line
segment $ \{ k \Delta x \} \times ( (n-1) \Delta y, (n+a) \Delta y)$,
a combined strong $2$--vortex sheet/$3$--entropy wave emanating from $((k-1) \Delta x, n \Delta y)$ meets $x = k \Delta x$
in the line segment $ \{ k \Delta x \} \times ( (n+a) \Delta y, (n+b) \Delta y)$, and
a strong $4$--shock emanating from $((k-1) \Delta x, n \Delta y)$ meets $x = k \Delta x$
in the line segment $ \{ k \Delta x \} \times ( (n+b) \Delta y, (n+1) \Delta y)$.

Now define
\[
a_{k,n} = (n + 1 + \theta_k) \Delta y,
\qquad k \in \N_0,\,\, n \in \mathbb{Z}, \,\, n + k \equiv 1 \, (\mathrm{mod } \, 2),
\]
where $\theta_k$ is randomly chosen in $(-1,1)$.
We then choose
\[
P_{k,n} = (k \Delta x, a_{k,n}), \qquad k \in \N_0, \,\, n \in \mathbb{Z},\,\, n + k \equiv 1\, (\mathrm{mod} \, 2),
\]
to be the mesh points, and define the approximate solutions $U_{\Delta x, \theta}$ globally in $\Omega_{\Delta x}$
for any $\theta = (\theta_0,\theta_1,\ldots)$ in the following inductive way:

To avoid the issues of interaction of strong fronts, we separate out the initial data for $k=0$.
Define $U_{\Delta x, \theta}\mid_{x=0}$ as follows:

\smallskip
First, in $\Omega_{\Delta x, 0} \setminus \{ -4 \Delta y \leq y \leq  4 \Delta y \}$, define
\[
U_{0, \Delta x, \theta}(y)
:= \begin{cases}
U_0( (2n - 3 + \theta_0) \Delta y) & \quad\mbox{for $y \in (2n \Delta y, (2n+2) \Delta y),   n \geq 2$}, \\[2mm]
U_0( (2n + 3 + \theta_0) \Delta y) & \quad \mbox{for $y \in (2n \Delta y, (2n+2) \Delta y),  n \leq -2$}.
\end{cases}
\]
Then define $U_0(y)$ to be the state that connects to $U_0(-4\Delta y)$ by a strong $1$--shock of
strength $\sigma_{10}$ on $(-4 \Delta y, 0)$ which lies in $O_\epsilon(U_{m1})$,
and the state that connects to $U_0(0)$ by a combined strong vortex sheet/entropy wave of strength
$(\sigma_{20},\sigma_{30})$ on $(0, 4 \Delta y)$ which lies in $O_\epsilon(U_{m2})$.
Thus, for $\epsilon$ small,
$U_{0,\Delta x,\theta}(y) \in O_{\epsilon}(U_a)$ for $y > 4 \Delta y$,
and $U_{0,\Delta x, \theta}(y) \in O_{\epsilon}(U_b)$ for $y < - 4 \Delta y$.

Now, assume that $U_{\Delta x, \theta}$ on $\{ 0 \leq x \leq k \Delta x \}$ has been defined for $k \geq 0$.
We solve the family of Riemann problems for $n \in \mathbb{Z}$
with $k + n \equiv 0 \,\, (\mathrm{mod} \, 2)$:
\[
\begin{cases}
W(U_n)_x + H(U_n)_y = 0 &\quad  \text{in } T_{k,n}, \\
U_n\mid_{x = k \Delta x} = U^0_{k,n}
\end{cases}
\]
with
\[
T_{k,n} = \{ k \Delta x \leq x \leq (k+1) \Delta x, (n-1)\Delta y \leq y \leq (n+1) \Delta y \},
\]
then set
\[
U_{\Delta x, \theta} = U_n \qquad \text{on } T_{k,n}, n \in \N,
\]
and define
\[
U_{\Delta x, \theta} \mid_{x = (k+1)\Delta x} = U_{\Delta x, \theta}((k+1)\Delta x - ,a_{k+1,n}),
\qquad k+n \equiv 0 \,\, (\mathrm{mod} \, 2).
\]

Now, as long as we can provide a uniform bound on the solutions and show the Riemann problems involved always have solutions,
this algorithm defines a family of approximate solutions globally.
\begin{figure}[h!]
	\includegraphics[scale=0.75]{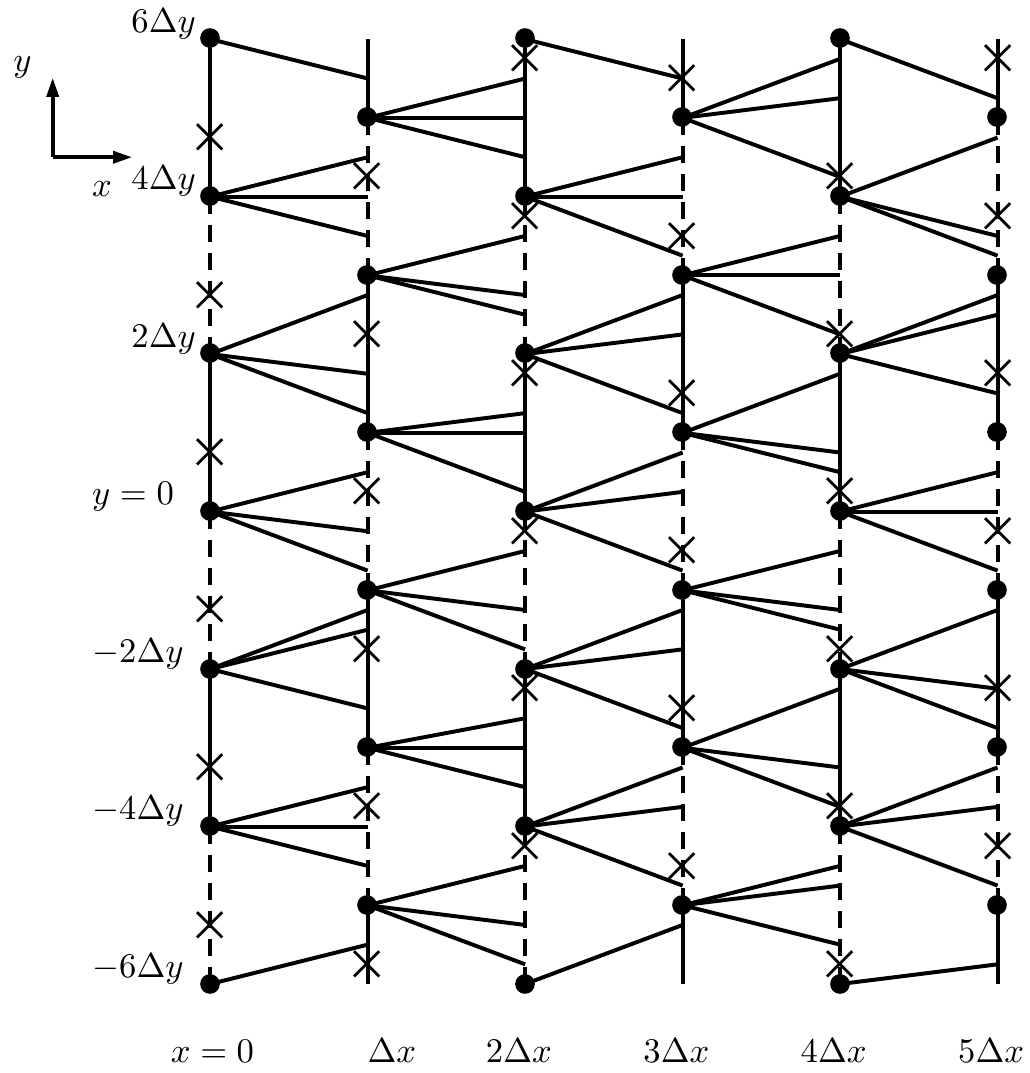}
	\centering
	\caption{Glimm's scheme}
	\label{fig:glimmdiagram}
\end{figure}

\subsection{Glimm-type functional and its bounds}
In this section, we prove the approximate solutions are well defined in $\Omega_{\Delta x}$ by uniformly bounding them.
We first introduce the following lemma,
which is a combination of  Lemma 6.7 in \cite[pp. 305]{shockwedges} and Lemma 4.1 in \cite[pp. 1677]{compvortex},
and follows from that $\Phi$, $G_1$, $G_{2,3}$, and $G_4$ are $C^1$ functions.

\begin{lemma} \label{lemma:bounds}
The following bounds of the approximate solutions of the Riemann problems hold{\rm :}
\begin{itemize}
\item If $\{ U_1, U_2 \} = (\alpha_1,\alpha_2, \alpha_3,\alpha_4)$ with $U_1,U_2 \in O_{\epsilon}(U_i)$ for fixed $i \in \{ a, m1,m2, b \}$,
then
\[
| U_1 - U_2| \leq B_1 ( |\alpha_1| + |\alpha_2| + |\alpha_3| + |\alpha_4|)
\]
with
\[
B_1 = \max_{i \in \{a, m1, m2, b \}, 1 \leq j \leq 4}
\Big\{\sup_{U \in O_{\epsilon}(U_i)} |\partial_{\alpha_j} \Phi(\alpha_4,\alpha_3,\alpha_2,\alpha_1;U)|\Big\}.
\]

\item For any $\sigma_j \in O_{\hat \epsilon}(\sigma_{j0})$, $j=2,3$, so that
\[
G_{2,3}(\sigma_3,\sigma_2,U) \subset  O_{\epsilon}(U_{m2}),
\]
then, when $U \in O_{\epsilon}(U_{m1})$ for some $\hat \epsilon = \hat \epsilon(\epsilon)$
with $\hat \epsilon \rightarrow 0$ as $\epsilon \rightarrow 0$,
\[
|G_{2,3}(\sigma_3,\sigma_2,U) -G_{2,3}(\sigma_{30},\sigma_{20},U)| \leq B_2\brac*{|\sigma_3 - \sigma_{30}| + |\sigma_2 - \sigma_{20}|}
\]
with
\[
B_2 = \max_{j=2,3} \Big\{\sup_{\sigma_j \in O_{\hat \epsilon}(\sigma_{j0})} |G'_{\sigma_j}(\sigma_3,\sigma_2,U)|\Big\}.
\]

\item For $\sigma_1 \in O_{\hat \epsilon}(\sigma_{10})$,
$G_1(\sigma_{\check \epsilon}(\sigma_{10}),U) \subset O_{\epsilon}(U_{m1})$ when $U \in O_{\epsilon}(U_b)$, and
\[
|G_1(\sigma,U) - G_1(\sigma_{10},U)| \leq B_3|\sigma - \sigma_{10}|,
\]
where $B_3 = \sup_{\sigma_1 \in O_{\check \epsilon}(\sigma_{10})} |\partial_{\sigma_1} G_1(\sigma_1, U)| $.
		
\smallskip
\item For $\sigma_4 \in O_{\hat \epsilon}(\sigma_{40})$,
$\tilde G_4(\sigma_{\check \epsilon}(\sigma_{40}),U) \subset O_{\epsilon}(U_{m2})$ when $U \in O_{\epsilon}(U_a)$, and
\[
|\tilde G_4(\sigma,U) - \tilde G_4(\sigma_{10},U)| \leq B_3|\sigma - \sigma_{40}|,
\]
where $B_4 = \sup_{\sigma_4 \in O_{\check \epsilon}(\sigma_{40})} |\partial_{\sigma 4} G_1(\sigma_4, U)| $.
\end{itemize}
\end{lemma}

We now show that $U_{\Delta x, \theta}$ can be defined globally.
Assume that $U_{\Delta x, \theta}$ has been defined in $\{x < k \Delta x \}$, $k \geq 1$,
by the steps in \S5,
and assume the following conditions are satisfied:
\begin{description}
\item[$\mathbf{C_1(k-1)}$]  In each $\Omega_{\Delta x, j}$, $0 \leq j \leq k-1$,
there are a strong $1$--shock $\chi_1^{(j)}$, a combined strong vortex sheet/entropy wave $\chi_{2,3}^{(j)}$,
and a strong $4$--shock $\chi_4^{(j)}$
in  $U_{\Delta x, \theta}$ with strengths $\sigma_1^{(j)}, (\sigma_2^{(j)},\sigma_3^{(j)})$, and $\sigma_4^{(j)}$
so that   $\sigma_i^{(j)} \in O_{\epsilon}(\sigma_{i0})$,
which divide   $\Omega_{\Delta x, j}$ into four parts{\rm :}
$\Omega^{(1)}_{\Delta x, j}$ -- the part below $\chi_1^{(j)}$,
$\Omega^{(2)}_{\Delta x, j}$ -- the part between $\chi^{(j)}_1$ and $\chi^{(j)}_{2,3}$,
$\Omega^{(3)}_{\Delta x, j}$ -- the part between $\chi^{(j)}_{2,3}$ and $\chi^{(j)}_{4}$,
and $\Omega^{(4)}_{\Delta x, j}$ -- the part above $\chi^{(j)}_{4}${\rm ;}

\smallskip
\item[$\mathbf{C_2(k-1)}$] For $0 \leq j \leq k-1$,
\[
\chi^{(j)}_{1}(j \Delta x) + 4 \Delta y \leq \chi^{(j)}_{2,3}(j \Delta x) \leq \chi^{(j)}_4(j \Delta x);
\]

\smallskip
\item[$\mathbf{C_3(k-1)}$] For $0 \leq j \leq j-1$,
\begin{align*}
& U_{\Delta x, \theta} \mid_{\Omega^{(1)}_{\Delta x, j}} \in O_{\epsilon}(U_b),
& &U_{\Delta x, \theta} \mid_{\Omega^{(2)}_{\Delta x, j}} \in O_{\epsilon}(U_{m1}),\\
& U_{\Delta x, \theta} \mid_{\Omega^{(3)}_{\Delta x, j}} \in O_{\epsilon}(U_{m2}),
& & U_{\Delta x, \theta} \mid_{\Omega^{(4)}_{\Delta x, j}} \in O_{\epsilon}(U_a).
\end{align*}
\end{description}
To see that $C_2(k)$ holds if $C_2(k-1)$ holds,
by the discussion earlier, for $\chi_{4}^{(k)}$ to emanate further down than $\chi_{4}^{(k-1)}$,
we require $\theta_k > b$,
which implies that $\chi_{2,3}^{(k)}$ and $\chi_1^{(k)}$ emanate further down than $\chi_{2,3}^{(k-1)}$
and $\chi_1^{(k-1)}$.
Also, for $\chi_{1}^{(k)}$ to emanate further up than $\chi_{1}^{(k-1)}$,
we require $\theta_k < a$, which implies that $\chi_{2,3}^{(k)}$ and $\chi_1^{(k)}$ emanate further
up than $\chi_{2,3}^{(k-1)}$ and $\chi_{2,3}^{(k-1)}${\rm ;}
see Fig. \ref{fig:furtherdown} for an illustration of the first situation.

\begin{figure}[h!]
	\includegraphics[scale=0.70]{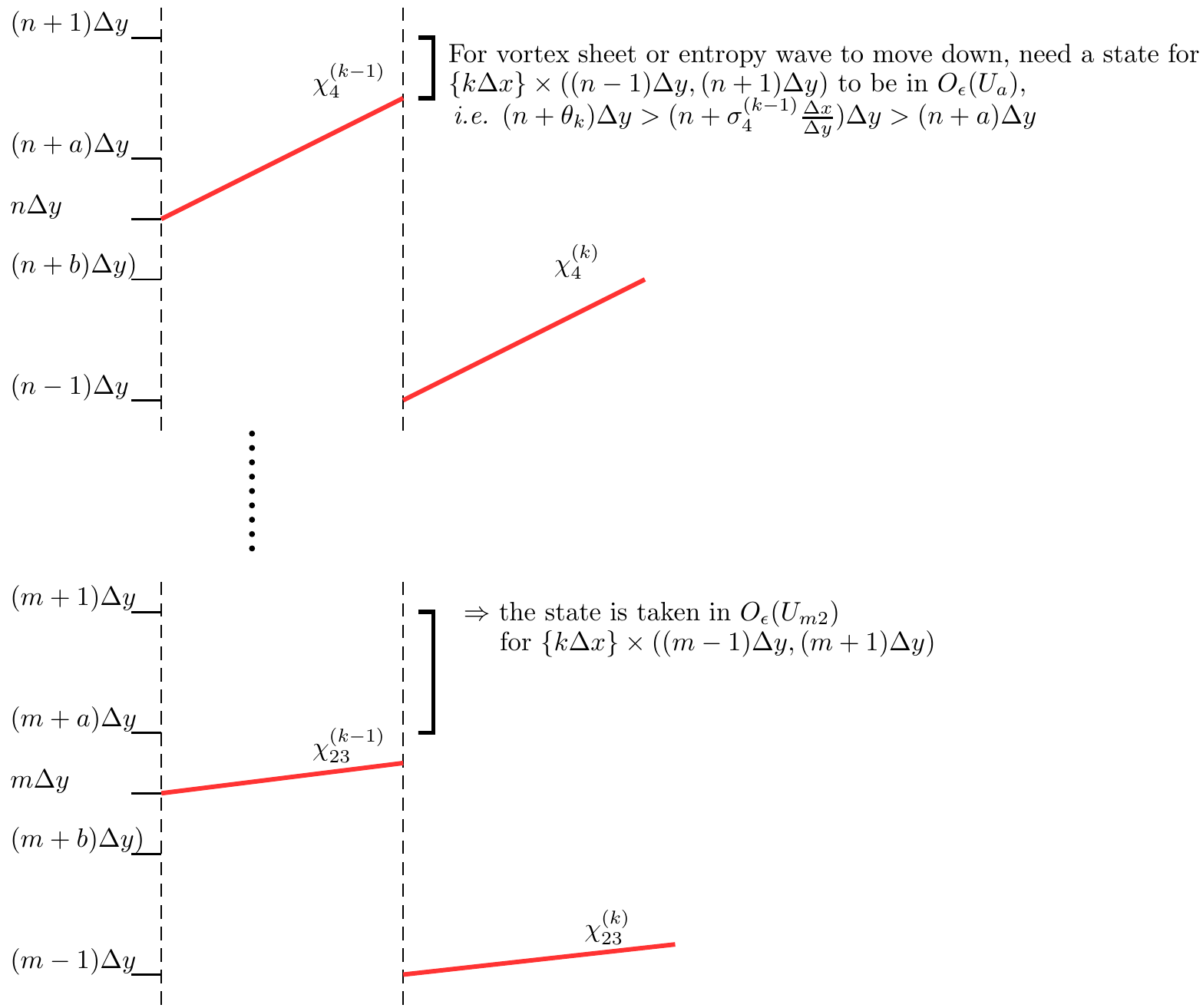}
	\centering
	\caption{Separation of the initial data}
	\label{fig:furtherdown}
\end{figure}

We will establish a bound on the total variation of $U_{\Delta x, \theta}$ on the $k$-mesh curves to establish $C_3(k)$ and $C_1(k)$.

\begin{definition}
A $k$-mesh curve is an unbounded piecewise linear curve consisting of line segments between the mesh points,
lying in the strip:
\[
(k-1) \Delta x \leq x \leq (k+1) \Delta x
\]
with each line segment of form $P_{k-1,n-1} P_{k,n}$ or $P_{k,n} P_{k+1,n+1}$.
\end{definition}

Clearly, for any $k>0$, each $k$-mesh curve $I$ divides plane $\R^2$ into a part $I^+$ and a part $I^-$,
where $I^-$ is the one containing set $\{ x < 0 \}$.
As in \cite{glimm}, we partially order these mesh curves by saying $J > I$ if every point of $J$ is either on $I$
or contained in $I^+$,
and we call $J$ an immediate successor to $I$ if $J > I$ and every mesh point of $J$,
except one, is on $I$. We now define a Glimm-type functional on these mesh curves.

\begin{definition}
Define
\begin{align*}
			F_s(J) &= C^* \sum_{j=1}^4 |\sigma_i^J - \sigma_{i0}| + F(J)  + K Q(J),  \\
			F(J) &= L(J) + \tilde K_{43} |\alpha_{4s1}| + \tilde K_{12} |\alpha_{1s4}|,  \\[2mm]
			L(J) &= L^1(J) + L^2(J) + L^3(J) + L^4(J), \\
			L^j(J) &= \sum_{j=1}^4 K^*_{ij} L_i^j(J), \\
			L^j_i(J) & = \sum \cbrac*{ |\alpha_i| \, :\, \alpha_i \text{ crosses } J \text{ in region } (j) }, \\[2mm]
			Q(J) &= \sum \{ |\alpha_i| |\beta_j | \, : \, \text{ both } \alpha_i \text{ and } \beta_j \text{ cross } J \text{ and approach}\},
\end{align*}
where $|\alpha_{4s1}|$ and $|\alpha_{1s4}|$ are the strengths of the weak $4$--shock and $1$--shock that emanate from the same point
as the strong $1$--shock and $4$--shock, respectively,
and these two weak waves are excluded from $L_4^3$ and $L_1^2$, respectively.
\end{definition}
	
\begin{remark}
$L(J) + C^* \sum_{j=1}^4 |\sigma_i^J - \sigma_{i0}|$ measures the total variation of $\left. U\right|_{J}$,
owing to Lemma $\ref{lemma:bounds}${\rm ;}
each term $L^i(J)$ measures the total variation in region $(i)$,
and each term $|\sigma_i^J - \sigma_{i0}|$ measures the magnitude of jumps of $U$ between the regions separated
by the large shocks.
We also see that $Q(J) \leq L(J)$ for $L(J)$ small enough,
so that $Q(J)$ is equivalent to $F_s(J) + C^* \sum_{j=1}^4 |\sigma_i^J - \sigma_{i0}|$.
\end{remark}
	
\begin{proposition}\label{prop:functional}
Suppose that $I$ and $J$ are two $k$--mesh curves such that $J$ is an immediate successor of $I$.
Suppose that
\begin{align*}
&\big|U_{\Delta x, \theta} |_{I \cap (\Omega^{(1)}_{\Delta x, k-1} \cup \Omega^{(1)}_{\Delta x, k})} - U_b\big| < \epsilon,
			\qquad\,\,\,\,	\big|U_{\Delta x, \theta}|_{I \cap (\Omega^{(2)}_{\Delta x, k-1} \cup \Omega^{(2)}_{\Delta x, k})} - U_{m1}\big| < \epsilon, \\[1mm]
&  \big|U_{\Delta x, \theta}|_{I \cap (\Omega^{(3)}_{\Delta x, k-1} \cup \Omega^{(3)}_{\Delta x, k})} - U_{m2}\big| < \epsilon,
			  \qquad	\big|U_{\Delta x, \theta}|_{I \cap (\Omega^{(4)}_{\Delta x, k-1} \cup \Omega^{(4)}_{\Delta x, k})} - U_a\big| < \epsilon, \\[2mm]
&	|\sigma^I_j - \sigma_{j0}| < \hat \epsilon(\epsilon), \qquad j = 1,2,3,4,
\end{align*}
for some $\hat \epsilon( \epsilon) > 0$ determined in Lemma {\rm \ref{lemma:bounds}}.
Then there exists $\tilde \epsilon > 0$ such that, if $F_s(I) \leq \tilde \epsilon$,
\[
F_s(J) \leq F_s(I),
\]
and hence
\begin{align*}
&\big|U_{\Delta x, \theta}|_{J \cap (\Omega^{(1)}_{\Delta x, k-1} \cup \Omega^{(1)}_{\Delta x, k})} - U_b\big| < \epsilon,
   \quad \,\,\,\,\big|U_{\Delta x, \theta}|_{J \cap (\Omega^{(2)}_{\Delta x, k-1} \cup \Omega^{(2)}_{\Delta x, k})} - U_{m1}\big| < \epsilon, \\[1mm]
&	\big|U_{\Delta x, \theta}|_{J \cap (\Omega^{(3)}_{\Delta x, k-1} \cup \Omega^{(3)}_{\Delta x, k})} - U_{m2}\big| < \epsilon,
   \quad \big|U_{\Delta x, \theta}|_{J \cap (\Omega^{(4)}_{\Delta x, k-1} \cup \Omega^{(4)}_{\Delta x, k})} - U_a\big| < \epsilon, \\[2mm]
&	|\sigma^J_j - \sigma_{j0}| < \hat \epsilon(\epsilon) \qquad \mbox{for $j = 1,2,3,4$}.
\end{align*}
\end{proposition}

\begin{proof}
		\begingroup
		\setlength{\abovedisplayskip}{2pt}
		\setlength{\belowdisplayskip}{2pt}
We make $M\ge 2$ larger to ensure that it is bigger than all the $O(1)$--terms in \S \ref{chapter:waves},
and that $M\ge \frac{\Delta y}{\Delta x}$.
With the fixed $M$ from here on, we now define our constants in terms of it.
We set
\[
\tilde \epsilon = \frac{1}{1024M^3}
\]
and
\begin{align*}
&K = 16M, \qquad C^* = \frac{1}{128M^2}, \\
&K^*_{11} = K^*_{21} = K^*_{31} = K^*_{41} = K^*_{14} = K^*_{24} = K^*_{34} = K^*_{44} =1, \\
&K^*_{12} = K^*_{43} = \frac{1}{M}, \quad  K^*_{13} = K^*_{42} = 1+\frac{2}{M}, \\
&\tilde K^*_{42} = \tilde K^*_{13} = \frac{1}{16M}, \quad K^*_{22} = K^*_{32} = K^*_{2,3} = K^*_{33} = \frac{1}{32M}.
\end{align*}
Let $\Lambda$ be the diamond that is formed by $I$ and $J$.
We can assume that $I = I_0 \cup I'$ and $J = J_0 \cup J'$
such that $\partial \Lambda = I' \cup J'$.
We divide our proof into different cases, based on where diamond $\Lambda$ is located.

\smallskip		
\textbf{Case 1: Weak-weak interaction.}
Suppose that $\Lambda$ lies in the interior of a region $(i)$;
see Figure \ref{fig:weakwave}.
Then, by Proposition \ref{prop:weakwaves},
\begin{align*}
L^i(J) - L^i(I) & \leq \sum_{j=1}^4 K^*_{ij} M Q(\Lambda),
\end{align*}
and
\begin{align*}
		Q(J)-  Q(I) & = \Big(Q(I_0) + \sum_{i=1}^4Q(\delta_i,I_0)\Big) - \Big(Q(I_0) + Q(\Lambda) + \sum_{i=1}^4 Q(\alpha_i,I_0) + \sum_{i=1}^4 Q(\beta_i,I_0)\Big) \\
			& \leq \big(4M L(I_0) -1\big) Q(\Lambda) \\
			& \leq -\frac 12 Q(\Lambda),
\end{align*}
when $L(I_0) \leq \tilde \epsilon \leq \frac{1}{8M}$.
Since $K_{ij}^* < 2$ for all $i$ and $j$,
we have
$$
\sum_{j=1}^4 K^*_{ij} M  \leq 8M = \frac{K}{2},
$$
so that $F_s(J) \leq F_S(I)$.

\begin{figure}[h!]
\includegraphics[scale=0.7]{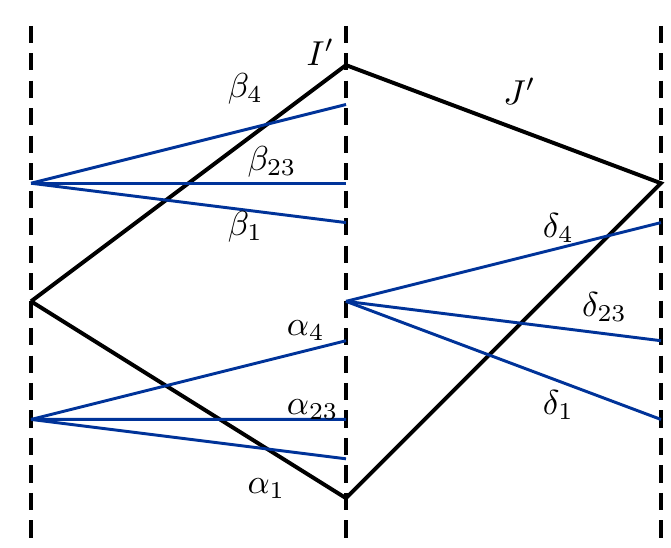}
\center                    ing
\caption{Case 1: Weak wave interaction}
\label{fig:weakwave}
\end{figure}
		
\textbf{Case 2: Weak waves interact with the strong vortex sheet/entropy wave from below.}
Suppose that the approximate strong vortex sheet/entropy wave enters $\Lambda$ from above; see Figure \ref{fig:belowvortex}.
\begin{figure}[h!]
			\includegraphics[scale=0.7]{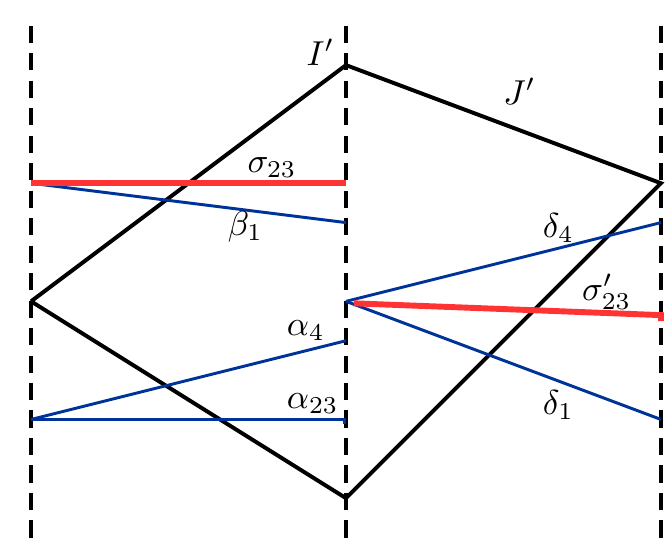}
			\includegraphics[scale=0.7]{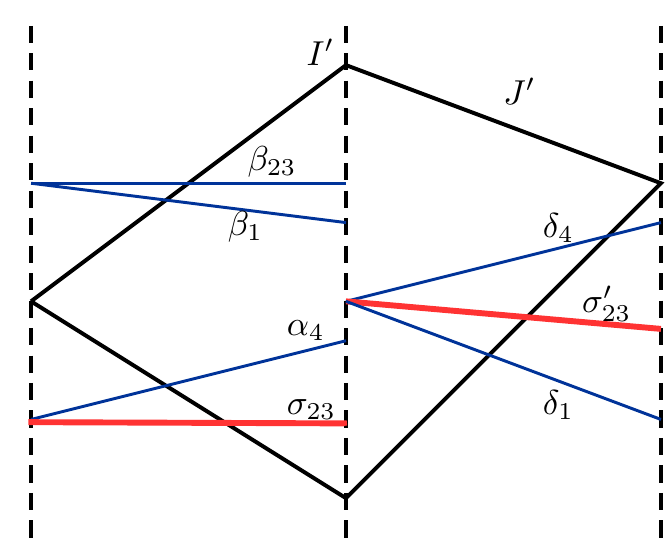}
			\centering
			\caption{Cases 2 \& 3: Weak waves interact with the strong vortex sheet/entropy wave from below and above}
			\label{fig:belowvortex}
			\label{fig:abovevortex}
\end{figure}
We have
\begin{align*}
			L^3_4(J) - L^3_4(I) &= |\delta_4| \leq M |\alpha_4| + M|\beta_1|(  |\alpha_2| + |\alpha_3|), \\
			L^2_1(J) - L_1^2(I) &= |\delta_1| - |\beta_1| \leq  |\alpha_4| + M|\beta_1|( +|\alpha_2| + |\alpha_3|), \\
			L^2_i(J) - L^2_i(I) &= -|\alpha_i| \qquad \mbox{for}\,\, i= 2,3,
\end{align*}
\begin{align*}
			Q(J) - Q(I) &= Q(I_0) + Q(\delta_4,I_0) + Q(\delta_1,I_0)\\
            &\quad - \big\{ Q(I_0) + |\beta_1|(|\alpha_2| + |\alpha_3|)+ Q(\beta_1,I_0) + Q(\beta_4,I_0)    \\
            &\quad\quad\,\,\,  + Q(\alpha_1,I_0) + Q(\alpha_2,I_0) + Q(\alpha_3,I_0) + Q(\alpha_4,I_0) \big\}\\
			& \leq (1+M)|\alpha_4| L(I_0) + \big(M L(I_0) -1\big) |\beta_1|\big( |\alpha_1| + |\alpha_2| + |\alpha_3|\big)\\
			& \leq 2M |\alpha_4| L(I_0) - \frac{1}{2} |\beta_1|\big( |\alpha_2| + |\alpha_3|\big),
\end{align*}
and 	
\begin{align*}
			|\sigma_i^J - \sigma_i^I| & \leq M|\alpha_4| + |\alpha_i| + M|\beta_1|\big(  |\alpha_2|+|\alpha_3|\big)
\qquad \mbox{for $i=2,3$}.
\end{align*}
Then we obtain
\begin{align*}
 F_s(J) - F_s(I)
 \leq\, & \big(M K_{43}^* + K_{12}^*  - K^*_{42} + 2MK L(I_0) + M  C^*\big) |\alpha_4|  \\
 & + \Big(M K_{43}^* + K_{12}^* + 2M C^* - \frac{K}{2}\Big)|\beta_1|\big(|\alpha_2| + |\alpha_3|\big)  \\
& + (C^* - K^*_{22} ) |\alpha_2| + (C^* - K^*_{32}) |\alpha_3| \\
 \leq &\, \Big(1 + \frac{2}{M} - (1+\frac{2}{M}) + \frac{1}{32M} + \frac{1}{128M}\Big) |\alpha_4| \\
 & + \sum_{i=1}^3 \Big(1 + \frac{1}{M} - 8M\Big)|\beta_1|\big( |\alpha_2| + |\alpha_3|\big)
 \leq 0.
\end{align*}
		
\textbf{Case 3: Weak waves interact with the strong vortex sheet/entropy wave from above.}
Suppose that the approximate vortex sheet/entropy wave enters $\Lambda$ from below;
see Figure \ref{fig:abovevortex}.
This case follows by symmetry from the above case, owing to the symmetry between
Propositions \ref{prop:belowvortex} and \ref{prop:abovevortex},
and the symmetry between the coefficients.

\smallskip		
\textbf{Case 4: Weak waves interact with the strong $1$--shock from above.}
Suppose that the strong $1$--shock enters $\Lambda$ from below; see Figure \ref{fig:above1shock}.
By Proposition \ref{prop:strong1aboveest}, we have
		\begin{figure}[h!]
			\includegraphics[scale=0.7]{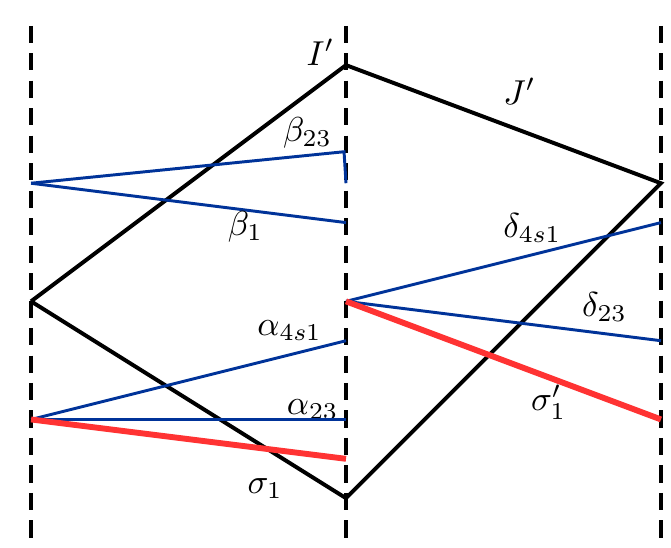}
			\includegraphics[scale=0.7]{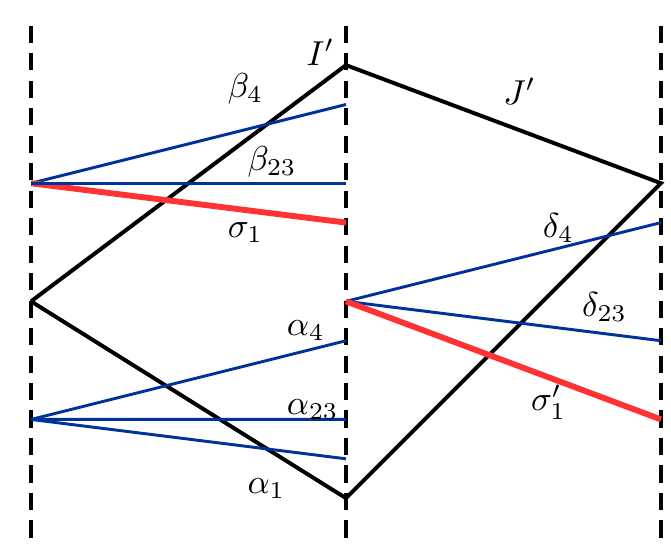}
			\centering
			\caption{Cases 4 \& 5: Weak waves interact with the strong $1$--shock from above and below}
			\label{fig:above1shock}
			\label{fig:below1shock}
		\end{figure}
\begin{align*}
			L_i^2(J) - L_i^2(I) &=|\delta_i| - |\beta_i| - |\alpha_i|
			\leq M |\beta_1| + M |\alpha_4|( |\beta_1| + |\beta_2| + |\beta_3| ) \qquad \mbox{for $i = 2,3$},\\
			L_j^2(J) - L_j^2(I)& = - |\beta_j| \qquad \mbox{for $j =1,4$},
\end{align*}
		and
		\begin{align*}
			Q(J) - Q(I) =&\, Q(I_0) + Q(\delta_2,I_0) + Q(\delta_3,I_0) + Q(\delta_4,I_0) \\
			& - \big\{ Q(I_0) + Q(\alpha_2,I_0) + Q(\alpha_3,I_0) + Q(\alpha_4,I_0)  \\
			& \,\,\quad+ Q(\beta_1,I_0) + Q(\beta_2,I_0) + Q(\beta_3,I_0) + Q(\beta_4,I_0)  \\
			&  \,\,\quad +|\alpha_{4s1}|( |\beta_1| + |\beta_2| + |\beta_3|  ) + |\alpha_3|(|\beta_1| + |\beta_2|) + |\alpha_2||\beta_1| \big\} \\[1mm]
			 \leq&\,  M L(I_0) |\beta_1| + (M L(I_0) - 1) |\alpha_4| ( |\beta_1| + |\beta_2| + |\beta_3|) \\ & -  |\alpha_3|(|\beta_1| + |\beta_2|) - |\alpha_2||\beta_1|  \\[1mm]
			 \leq&\, M L(I_0) |\beta_1|
           - \frac{1}{2} \big(\alpha_4| ( |\beta_1| + |\beta_2| + |\beta_3|) + |\alpha_3|(|\beta_1| + |\beta_2|) + |\alpha_2| |\beta_1| \big),
		\end{align*}

		\begin{align*}
			|\delta_{4s1}| - |\alpha_{4s1}| & \leq M|\beta_1| + M|\alpha_4|( |\beta_1| + |\beta_2| + \beta_3| ),
		\end{align*}
		and
		\begin{align*}
			|\sigma^J_1 - \sigma^I_1| & \leq M |\beta_1| + M |\alpha_4|( |\beta_1| + |\beta_2| + |\beta_3| ) .
		\end{align*}				

Therefore, we obtain
		\begin{align*}
			 F_s(J) - F_s(I)
            \leq &\, \big(M(K_{22}^* + K_{32}^* + \tilde K_{42}^*) + K M L(I_0) + C^*M - K_{42}^*\big) |\beta_1| \\
			  & + \big( - \frac{K}{2} + M(K_{22}^* + K_{32}^* + \tilde K_{42}^*) + C^*M\big) |\alpha_4|( |\beta_1| + |\beta_2| + |\beta_3|)  \\
			\leq &\, 0.
		\end{align*}

\smallskip
\textbf{Case 5: Weak waves interact with the strong $1$--shock from below.}
 Suppose that the strong $1$--shock enters $\Lambda$ from above; see Figure \ref{fig:below1shock}.
 By Proposition \ref{prop:strong1belowest}, we have

\begin{align*}
&L_i^2(J) - L_i^2(I)  \leq M\big(|\alpha_1| + |\alpha_2| + |\alpha_3| + |\alpha_4\big) \qquad \mbox{for $i = 2,3$},\\
&L_j^1(J) - L_j^1(I)  = - |\alpha_j| \qquad \mbox{for $j=1,2,3,4$},\\[3mm]
&|\delta_4| - |\beta_4|  \leq M(|\alpha_1| + |\alpha_2| + |\alpha_3| + |\alpha_4), \\
&|\sigma_1^J - \sigma_1^I|  \leq M(|\alpha_1| + |\alpha_2| + |\alpha_3| + |\alpha_4),
\end{align*}
\begin{align*}
Q(J) - Q(I) =&\, Q(I_0) + Q(\delta_1,I_0) + Q(\delta_2,I_0) + Q(\delta_3,I_0) + Q(\delta_4,I_0) \\
			&\, - \big\{ Q(\alpha_1,I_0) + Q(\alpha_2,I_0) + Q(\alpha_3,I_0) + Q(\alpha_4,I_0) \\
			&\, \quad\,\, + Q(\beta_2,I_0) + Q(\beta_3,I_0) + Q(\beta_4,I_0) \\
			&\,  \quad\,\, + |\alpha_4|( |\beta_2| + |\beta_3| + \mathbf{1}_{ ([0,\infty)^2)^c}(\alpha_4,\beta_4) |\beta_4|) + |\alpha_3||\beta_2| \big\} \\
\leq & \, M L(I_0)(|\alpha_1| + |\alpha_2| + |\alpha_3| + |\alpha_4|) \\
     &\, - \big(|\alpha_4|( |\beta_2| + |\beta_3| + \mathbf{1}_{ ([0,\infty)^2)^c}(\alpha_4,\beta_4) |\beta_4|) + |\alpha_3||\beta_2|\big).
		\end{align*}
Then
\begin{align*}
 F_s(J) - F_s(I) \leq &\, \sum_{j=1}^4 \brac*{ M(K^*_{22} + K^*_{32} + \tilde K^*_{42}) + K ML(I_0) + M C^* - K_{j1} }|\alpha_j| \\
			 \leq & \,  - \frac{1}{32}\sum_{j=1}^4 |\alpha_j|  \leq  0.
\end{align*}

\textbf{Case 6: Weak waves interact with the strong $4$--shock from below.}
Suppose that the strong $4$--shock enters $\Lambda$ from above; see Figure \ref{fig:below4shock}.
By symmetry from Case 4, we conclude that $F_s(J) \leq F_s(I)$, due to the symmetry
between Propositions \ref{prop:strong1aboveest} and \ref{prop:strong4belowest}.
		\begin{figure}[h!]
			\includegraphics[scale=0.7]{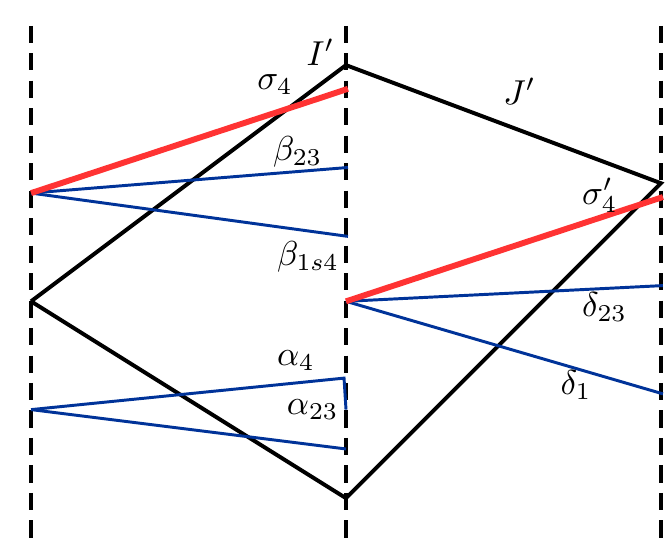}
			\includegraphics[scale=0.7]{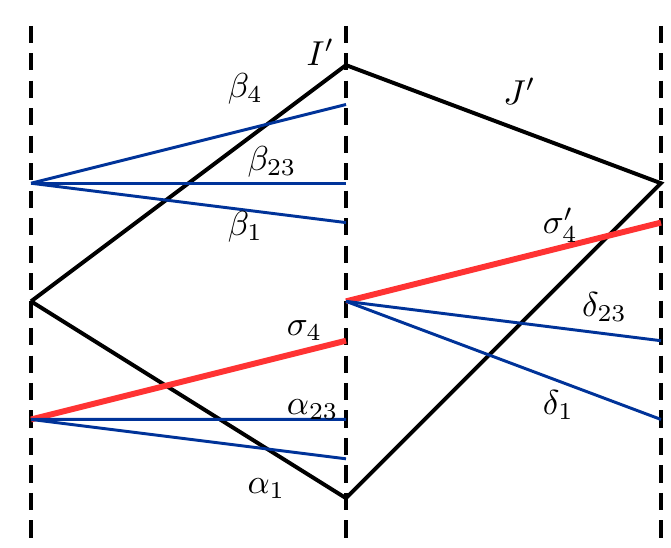}
			\centering
			\caption{Cases 6 \& 7: Weak waves interact with the strong $4$--shock from below and above.}
			\label{fig:below4shock}
			\label{fig:above4shock}
		\end{figure}

\textbf{Case 7: Weak waves interact with the strong $4$--shock from above.}
Suppose that the strong $4$--shock enters $\Lambda$ from below; see Figure \ref{fig:above4shock}.
By symmetry from Case 5, we obtain that  $F_s(J) \leq F_s(I)$, owing to the symmetry between
Propositions \ref{prop:strong1belowest} and \ref{prop:strong4aboveest}.
\endgroup
\end{proof}

Let $I_k$ be the $k$--mesh curve lying in $\{ (j-1)\Delta x  \leq x \leq j \Delta x \}$.
From Proposition \ref{prop:functional}, we obtain the following theorem for any $k \geq 1$:
	
\begin{theorem}\label{thm:decreasingtv}
Let $\tilde \epsilon, \hat \epsilon(\epsilon), K$, and $C^*$ be the constants specified
in Proposition {\rm \ref{prop:functional}}.
If the induction hypotheses $C_1(k-1)- C_4(k-1)$ hold and $F_s(I_{k-1}) \leq \tilde \epsilon$,
then
\begin{equation}\label{eq:solnbounds}
\begin{aligned}
&\big|U_{\Delta x, \theta} |_{\Omega^{(1)}_{\Delta x, k}} - U_b\big| < \epsilon,
			\qquad\,\,\,\,	\big|U_{\Delta x, \theta} |_{\Omega^{(2)}_{\Delta x, k}} - U_{m1}\big| < \epsilon,  \\
			&	\big|U_{\Delta x, \theta} |_{\Omega^{(3)}_{\Delta x, k}} - U_{m2}\big| < \epsilon,
			\qquad	\big|U_{\Delta x, \theta}\mid_{\Omega^{(4)}_{\Delta x, k}} - U_a\big| < \epsilon, \nonumber \\
			&	|\sigma^{(k)}_j - \sigma_{j0}| < \hat \epsilon(\epsilon)\qquad \mbox{for $j = 1,2,3,4$}, \nonumber
\end{aligned}
\end{equation}
		and
\begin{equation}
 F_s(I_k) \leq F_s(I_{k-1}).
\end{equation}
\end{theorem}

Moreover, we obtain the following theorem by the construction of our approximate solutions:
	
\begin{theorem} \label{thm:vortexandtv}
There exists $\epsilon > 0$ such that, if
\[
\mathrm{TV}\brac*{U_0(\cdot) - \overline U(0,\cdot) } < \epsilon,
\]
then, for any $\theta \in \prod_{k=0}^\infty (-1,1)$ and every $\Delta x > 0$,
the modified Glimm scheme defines a family of strong approximate fronts
$\chi_{j,\Delta x, \theta}$, $j=1, 2,3, 4$,
in $\Omega_{\Delta x, \theta}$ which satisfy $C_1(k-1) - C_4(k-1)$ and \eqref{eq:solnbounds}.
In addition,
\[
\abs*{ \chi_{j,\Delta x, \theta}(x+h) - \chi_{j, \Delta x, \theta}(x)} \leq \brac*{|\sigma_{j0}| + 2M}|h| + 2 \Delta y
\]
for any $x \geq 0$ and $h>0$.
\end{theorem}

\subsection{Estimates on the Approximate Strong Fronts}
Let
\begin{align}
& \sigma_{j,\Delta x, \theta}(x) = \sigma^{(k)}_j \qquad \text{for $x \in (k \Delta x, (k+1)\Delta x]$ and $j = 1,2,3,4$},\\
& s_{i,\Delta x, \theta}(x) = \sigma_{i,\Delta x, \theta}(x) \qquad\mbox{for $i =1,4$},\\
& s_{2,3,\Delta x, \theta} = \left. \frac{v_{\Delta x, \theta}}{u_{\Delta x, \theta}}\right|_{\{y = \chi^{(k)}_{2,3}\}}
  \qquad \text{for $x \in (k \Delta x, (k+1)\Delta x]$}.
\end{align}
Then, by the estimates in Proposition \ref{prop:functional}, we have

\begin{proposition}\label{prop:sigmatv}
There exists $\tilde M$, independent of $\Delta x, \theta$, and $U_{\Delta x, \theta}$, such that
\[
TV \{ \sigma_{j, \Delta x, \theta} \mid [0,\infty) \} = \sum_{k=0}^\infty |\sigma_{(k+1)} - \sigma_{(k)}|  \leq \tilde M.
\]
\end{proposition}

\begin{proof}
We follow the proof in \cite[pp. 290]{shockwedges}. For any $k \geq 1$, and the interaction diamond $\Lambda \subset \{ (k-1)\Delta x \leq x \leq (k+1) \Delta x \}$,
define 	
	\begin{equation}
		E_{\Delta x, \theta}(\Lambda) = \begin{cases}
			0 & \text{Case } 1, \\
			|\alpha_2| + |\alpha_3| + |\alpha_4| & \text{Case } 2, \\
			|\beta_1|+ |\beta_2| + |\beta_3|  & \text{Case } 3, \\
			|\beta_1| & \text{Case } 4, \\
			|\alpha_1| + |\alpha_2| + |\alpha_3| + |\alpha_4| &\text{Case } 5, \\
			|\alpha_4| & \text{Case } 6, \\
			|\beta_1| + |\beta_2| + |\beta_3| +|\beta_4| & \text{Case } 7.
		\end{cases}
	\end{equation}
We see  from Proposition \ref{prop:functional} that
\[
\sum_{\Lambda} E_{\Delta x, \theta} \leq \sum_{\Lambda} \frac{1}{\tilde \epsilon}  \big(F_s(I) - F_s(J)\big) \leq \frac{1}{\tilde \epsilon}F_s(0) = M_1.
\]
Since the quadratic terms from the interacting waves can always be dominated by $F_S(J)$, as shown in \S \ref{chapter:waves},
we have
\[
\sum_{k=0}^\infty |\sigma^{(k+1)}_j - \sigma^{(k)}_j| \leq M \sum_{\Lambda}  E_{\Delta x, \theta}(\Lambda)  \leq M_1 := \tilde M.
\]
\end{proof}

Therefore, we also conclude the following:

\begin{proposition}\label{prop:tvspeeds}
There exists $\tilde{\tilde{M}}$, independent of $\Delta x,  \theta$, and $U_{\Delta x, \theta}$, such that, for $j=1,2,3,4$,
\[
\mathrm{TV} \{ s_{j,\Delta x, \theta}(\cdot)\, \mid \,  [0,\infty) \} \leq \tilde{\tilde{M}}.
\]
\end{proposition}

\begin{proof}
For $j=1,4$, this follows immediately from Proposition \ref{prop:sigmatv}.
For $j=2,3$, observe that, if we make $O_{\epsilon}(U_{m1})$ small enough so that $\frac{u_{m1}}{2} < u < 2 u_{m1}$ and $-1 < v < 1$
for any $U= (u,v,p,\rho) \in O_\epsilon(U_{m1})$,
we have
\[
\mathrm{TV} \{ s_{j,\Delta x, \theta}(\cdot)  \,\mid \, [0,\infty) \}
\leq \frac{4}{u_{m2}^2}\big(2 u_{m2} + 1\big) \mathrm{TV}(U_{\Delta x, \theta}) \leq C\, \mathrm{TV}(U_0).
\]
\end{proof}

\section{Global Entropy Solutions}
In this section, we show the convergence of the approximate solutions to an entropy
solution close to the four-wave configuration solution.

\subsection{Convergence of the approximate solutions}

\begin{lemma}\label{lem:l1bound}
For any $h > 0$ and $x \geq 0$,
there exists a constant $\check N$, independent of $\Delta x, \theta$, and $h$, such that
\[
\int_{-\infty}^\infty \abs*{U_{\Delta x, \theta}(x+h,y) - U_{\Delta x, \theta}(x,y)} \, dy
\leq \check N |h|.
\]
\end{lemma}

\begin{proof}
By Fubini's theorem and Theorem \ref{thm:vortexandtv},
\begin{align*}
\int_{-\infty}^\infty \abs*{ U_{\Delta x, \theta}(x+h,y) - U_{\Delta x, \theta(x,y)} } \, dy
& \leq \int_{-\infty}^\infty \int_{x}^{x+h} |d(U_{\Delta x, \theta}(s,y))| \, dy \\[1mm]
& = \int_{x}^{x+h} \mathrm{TV} (U_{\Delta x, \theta}(s, \cdot)) \, ds \\[1mm]
&\leq  \check N |h| \mathrm{TV}(U_0).
\end{align*}
\end{proof}

Since  $U_{\Delta x, \theta}$ is an entropy solution in each square $T_{k,n}, k +n \equiv 0\, (\mathrm{mod} \, 2)$,
we see that, for each test function $\phi \in C^\infty_0(\R^2;\R^4)$,
\begin{align*}
&\int_{k \Delta x}^{(k+1)\Delta x} \int^{(n+1) \Delta y}_{(n-1) \Delta y}
\big(\partial_x \phi \cdot H(U_{\Delta x,\theta}) + \partial_y \phi  \cdot W(U_{\Delta x, \theta})\big)\, dy \, dx  \\[2mm]
	&=\int_{k \Delta x}^{(k+1)\Delta x} \big(\phi(x, (n+1)\Delta y)\cdot W(U_{\Delta x, \theta}(x, (n+1) \Delta y))\\
&\qquad\qquad\qquad  - \phi(x, (n-1)\Delta y) \cdot W(U_{\Delta x,\theta}(x, (n-1) \Delta y))\big) \, dx\\[2mm]
	&\quad + \int_{(n-1)\Delta y}^{(n+1)\Delta y} \big(\phi( (k+1) \Delta x,y) \cdot H(U_{\Delta x, \theta}((k+1)\Delta x-, y)) \\
&\qquad\qquad\qquad\,\,\,\,\, - \phi(k \Delta x, y) \cdot H(U_{\Delta x, \theta}(k \Delta x+,y))\big)\, dx.
\end{align*}
Then, summing over all $k$ and $n$ with $k+n \equiv 1 \, (\mathrm{mod} \, 2)$ and re-arranging the terms, we have
\begin{align*}
	&\int_{0}^{\infty} \int^{\infty}_{-\infty}
         \big(\partial_x \phi \cdot  H(U_{\Delta x,\theta}) + \partial_y \phi\cdot W(U_{\Delta x, \theta})\big) \, dy \, dx  \\
	&= \sum_{k=1}^\infty \int_0^{\infty} \phi(k \Delta x,y)\cdot \big(H(U_{\Delta x, \theta}(k\Delta x-, y)) - H(U_{\Delta x, \theta}(k \Delta x+,y))\big)\, dx \\
	&\quad - \int_{-\infty}^\infty \phi(0,y) \cdot H(U_{\Delta x, \theta}(0+,y)) \, dy.
\end{align*}

Denote
\[
J(\theta,\Delta x, \phi)
= \sum_{k=1}^\infty \int_{-\infty}^\infty
\phi(k \Delta x, y) \cdot \big(H(U_{\Delta x, \theta}(k \Delta x+, y)) - H(U_{\Delta x, \theta}(k \Delta x -,y))\big) \, dy.
\]
Following the same steps as in \cite[Chapter 19]{smoller}, we have

\begin{lemma}\label{lem:convergence1}
There exist a null set $\mathcal{N} \subset \prod_{k=0}^\infty(-1,1)$ and a subsequence $\{ \Delta x_j \}_{j=1}^\infty \subset \{ \Delta x \}$,
which tends to 0, such that
\[
J(\theta, \Delta x_j, \phi) \rightarrow 0 \,\,\qquad \text{when } \Delta x_j \rightarrow 0
\]
for any $\theta \in \prod_{k=0}^\infty(-1,1) \setminus\mathcal{N}$ and $\phi \in C^\infty_0(\R^2;\R^4)$.
\end{lemma}

To complete the proof of the main theorem, we need to estimate the slope of the approximate strong fronts.
For $k \geq 1$, $i=1, 2, 3, 4$,
\[
d^i_k = s_i^{(k-1)} \frac{\Delta x}{\Delta y},
\]
with $d_k^{2,3}:=d^2_k=d^3_k$.
Then, by the choice of $\frac{\Delta y}{\Delta x}$,
we find that $d_k^i \in (-1,1)$, which depends only on $\{ \theta_l \}_{l=1}^{k-1}$.
We then define
\[
I^i(x , \Delta x, \theta) = \sum_{k=0}^{ \sqbrac*{x/\Delta x}} I^i_k(\Delta x, \theta),
\]
where
\[
I^i_k(\Delta x, \theta) = \mathbf 1_{(-1,d^i_k)}(\theta_k)(1 - d^i_k) \Delta y - \mathbf 1_{(d^i_k,1)}(\theta_k) (1 + d^i_k) \Delta y \qquad\mbox{for $k \geq 1$},
\]
and
\[
I^1_0(\Delta x,\theta) = - 4 \Delta y, \quad I^{2}_0(\Delta x, \theta) =I^{3}_0(\Delta x, \theta)= 0, \quad I^4_0( \Delta x, \theta) = 4 \Delta y.
\]
Then $I^i_k(\Delta x, \theta)$ is the jump of $y = \chi_i(x)$ at $x = k \Delta x$,
and is a measurable function of $(\Delta x, \theta)$, depending only on $U_{\Delta x, \theta} \mid_{ \{0 \leq x \leq k \Delta x \} }$ and $\{ \theta_l \}_{l=0}^k$.

\begin{lemma}\label{lem:formofshocks}  The following statements hold{\rm :}
\begin{enumerate}
\item[\rm (i)]
For any $x \geq 0$, $\Delta x \geq 0$, $\theta \in \prod_{k=0}^\infty (-1,1)$, and $i=1, 2,3, 4$,
\[
\chi_{i, \Delta x, \theta} = I^i(x, \Delta x, \theta) + \int_0^x s_{i, \Delta x, \theta}(s) \, ds;
\]
\item[\rm (ii)]
There exist a null set $\mathcal{N}_1$ and a subsequence $\{ \Delta_l \}_{l=1}^\infty \subset \{ \Delta x_j \}_{j=1}^\infty$ such that
\[
\int_0^\infty e^{-x} |I(x,\Delta_l, \theta)|^2 \, dx \rightarrow 0\,\, \qquad \text{when } \Delta_l \rightarrow 0
\]
for any $\theta \in \prod_{k=0}^\infty (-1,1) \setminus \mathcal{N}_1$.
\end{enumerate}
\end{lemma}

\begin{proof}
The first part is a direct calculation.
The second follows by the same proof as \cite[pp. 292]{shockwedges};
just take two sub-sub-sequences to obtain all three strong front slopes to converge.
\end{proof}

\begin{theorem}[Existence and Stability] \label{lem:existstabstrong}
There exist $\epsilon > 0$ and $C>0$ such that,
if the hypotheses of the main theorem hold,
then, for each $\theta \in \prod_{k=1}^\infty (-1,1) \setminus (\mathcal{N} \cup \mathcal{N}_1)$,
there exists a sequence $\{ \Delta_l \}_{l=1}^\infty$ of mesh sizes with $\Delta_l \rightarrow 0$ as $l \rightarrow \infty$,
and functions $U_\theta \in BV(\R^2_+; \R^4)$ and $\chi_{j,\theta} \in \mathrm{Lip}(\R_+;\R)$
with $\chi_{j,\theta}(0) = 0$, $j=1,2,3,4$, such that

\begin{enumerate}
\item[\rm (i)] $U_{\Delta l, \theta}$ converges to $U_{\theta}$ a.e. in $\R^2_+$,
   and $U_\theta$ is a global entropy solution of system \eqref{eq:euler} and satisfies the initial data \eqref{eq:initdata} a.e.{\rm ;}

\smallskip
\item[\rm (ii)] $\chi_{j,\Delta l, \theta}$ converges to $\chi_{j,\theta}$ uniformly in any bounded $x$-interval{\rm ;}

\smallskip
\item[\rm (iii)] $s_{j, \Delta l, \theta}$ converges to $s_{j,\theta} \in BV([0,\infty))$ a.e. and
\[
\chi_{j,\theta}(x) = \int_0^x s_{j,\theta}(t) \, dt.
\]
\end{enumerate}
\end{theorem}

\begin{proof}
Result (i) follows by the same steps as \cite[Chapter 19]{smoller},
(ii) follows by Theorem \ref{thm:vortexandtv} and the Arzela-Ascoli theorem,
while (iii)  follows Proposition \ref{prop:tvspeeds} and the basic properties of BV functions.
\end{proof}

\bigskip
%%%%%%%%%%%%%%%%%%%%%%%%%%%%%%%%%%%%%%%%%%%%%%%%%%%%%%%%%%%%%%%%%%%%%%%%%%%%%%%%%%%%%%%%%
\textbf{Acknowledgements}. $\,$  The research of
Gui-Qiang G. Chen was supported in part by
the UK Engineering and Physical Sciences Research Council Award
EP/E035027/1 and EP/L015811/1.
The research of Matthew Rigby was supported in part by
the UK Engineering and Physical Sciences Research Council Award
EP/E035027/1.

\bibliographystyle{plain}

\end{document}